\documentclass[11pt, a4paper]{article}
\usepackage{amsfonts, amsmath, amsthm, enumitem, mathtools, amssymb}

\usepackage{tikz-cd}

\theoremstyle{plain}
\newtheorem{theorem}{Theorem}[section]

\newtheorem{corollary}[theorem]{Corollary}
\newtheorem{lemma}[theorem]{Lemma}
\newtheorem{proposition}[theorem]{Proposition}
\newtheorem*{claim*}{Claim}
\newtheorem*{problem*}{Problem}
\newtheorem*{conjecture*}{Conjecture}

\theoremstyle{definition}
\newtheorem{definition}[theorem]{Definition}

\newtheorem{example}[theorem]{Example}
\newtheorem{remark}[theorem]{Remark}
\newtheorem{question}[theorem]{Question}

\newcommand\al{\alpha}
\newcommand\bt{\beta}
\newcommand\gm{\gamma}
\newcommand\dl{\delta}
\newcommand\lm{\lambda}
\newcommand\sg{\sigma}

\newcommand\cF{\mathcal{F}}
\newcommand\cJ{\mathcal{J}}
\newcommand\cM{\mathcal{M}}

\newcommand\la{\langle}
\newcommand\ra{\rangle}
\newcommand\lla{\langle\!\langle}
\newcommand\rra{\rangle\!\rangle}

\newcommand\ad{\mathrm{ad}}

\newcommand{\Sym}{\mathrm{Sym}}

\newcommand{\ch}{\mathrm{char}}

\newcommand\N{\mathbb{N}}
\newcommand\Z{\mathbb{Z}}
\newcommand\FF{\mathbb{F}}

\newcommand\Aut{\mathrm{Aut}}

\newcommand{\A}{\mathrm{A}}
\newcommand{\B}{\mathrm{B}}
\newcommand{\C}{\mathrm{C}}
\newcommand{\J}{\mathrm{J}}
\newcommand{\Y}{\mathrm{Y}}
\newcommand{\IY}{\mathrm{IY}}
\renewcommand{\L}{\mathrm{L}}

\newcommand{\Miy}{\mathrm{Miy}}
\newcommand{\AMiy}{\mathrm{AMiy}}

\newcommand{\Cl}{\widehat{S}(2)^\circ}

\newcommand{\cH}{\mathcal{H}}
\newcommand{\hatH}{\hat \cH}
\newcommand{\0}{\overline{0}}
\newcommand{\2}{\overline{2}}
\newcommand{\ii}{\overline{\imath}}
\newcommand{\jj}{\overline{\jmath}}
\renewcommand{\r}{\overline{r}}

\renewcommand{\t}{\overline{t}}

\newcommand{\sgn}{\mathrm{sgn}}

\newcommand{\im}{\mathrm{Im}}
\renewcommand{\phi}{\varphi}
\renewcommand{\epsilon}{\varepsilon}
\newcommand{\1}{{\bf 1}}

\newcommand{\ie}{i.e.\ }

\setlist[enumerate,1]{label={\upshape (\arabic*)}}
\setlist[enumerate,2]{label={\upshape (\alph*)}}
\setlist[enumerate,3]{label={\upshape (\roman*)}}

\usepackage{ltablex} 
\usepackage{tabularx}
\keepXColumns

\newcolumntype{C}[1]{>{\centering\arraybackslash}m{#1}}
\newcolumntype{Y}{>{\centering\arraybackslash}X}
\usepackage{xcolor}

\title{From forbidden configurations to a classification of some axial algebras of Monster type}
\author{J.~M\textsuperscript{c}Inroy\footnote{Department of Physical, Mathematics and Engineering Sciences, University of Chester, Exton Park, Parkgate Rd, Chester, CH1 4BJ, UK, and School of Mathematics, University of Bristol, Fry Building, Woodland Road, Bristol, BS8 1UG, UK, email: j.mcinroy@chester.ac.uk}
 \and
 S.~Shpectorov\footnote{School of Mathematics, University of Birmingham, Edgbaston, Birmingham, B15 2TT, UK, email: S.Shpectorov@bham.ac.uk}}

\date{}

\begin{document}
\maketitle

\begin{abstract}
Ivanov introduced the shape of a Majorana algebra as a record of the $2$-generated subalgebras arising in that algebra.  As a broad generalisation of this concept and to free it from the ambient algebra, we introduce the concept of an axet and shapes on an axet.  A shape can be viewed as an algebra version of a group amalgam. Just like an amalgam, a shape leads to a unique algebra completion which may be non-trivial or it may collapse. Then for a natural family of shapes of generalised Monster type we classify all completion algebras and discover that a great majority of them collapse, confirming the observations made in an earlier paper \cite{3gen4trans}.
\end{abstract}

\section{Introduction}

Axial algebras are a class of non-associative algebras which have a strong link to groups.  The motivating example is the Griess algebra whose axial behaviour was noticed in the context of vertex operator algebras arising in quantum physics.  Recently examples of axial algebras have been found in other areas of mathematics and beyond.  One such example is the class of Hsiang algebras in the theory of non-linear PDEs \cite{Tkachev}. Axial behaviour is also speculated by Fox in the algebras of vector flows on manifolds \cite{Fox}.

An axial algebra is generated by a set of special elements called axes. One way to describe the structure of an axial algebra is by its \emph{shape} which identifies for each pair of axes the $2$-generated subalgebra they generate. The Miyamoto group, which is a group of automorphisms naturally associated with the algebra, acts by permuting the set of axes $X$, and conjugate pairs of axes must generate isomorphic subalgebras. So the shape is consistent with respect to the action of the Miyamoto group. On the other hand, given a group $G$ acting on a set $X$ and a consistent choice $\cal S$ of subalgebras to pairs of elements from $X$, we can ask whether there is some algebra with $X$ as axes, $G$ as its Miyamoto group, and shape $\cal S$.  This is analogous to group amalgams where an amalgam is a configuration of groups and we ask whether there is a group (called a completion) which contains this configuration.

Shapes were first introduced for Majorana algebras by Ivanov in \cite{IPSS}. Initially, Majorana algebras were constructed for small hand-picked groups of interest and there would normally be one algebra for each shape. However, in a recent project \cite{3gen4trans}, where we systematically looked at all groups in a naturally defined class and all possible shapes, it transpired that the overwhelming majority of shapes collapse. That is, there is no non-trivial axial algebra with that shape. In fact, there we had over $11,000$ shapes and only about a hundred did not collapse. In \cite{3gen4trans}, we proposed a practical solution for how to eliminate collapsing shapes quickly. Namely, we identified a number of specific small collapsing shapes that frequently appear as subshapes of larger shapes, causing those larger shapes to collapse too. We call these small collapsing shapes \emph{forbidden configurations}.

This revelation means that the structure of an axial algebra for a given shape is far more restricted than one might otherwise think. So we have the two following theoretical problems, the first of which was posed in \cite{3gen4trans}.

\begin{problem*}
\begin{enumerate}
\item What natural conditions or additional axioms can we impose on the shapes to ensure that they lead to non-trivial algebras?
\item For a given shape, can we classify the non-trivial axial algebras with this shape?
\end{enumerate}
\end{problem*}

In this paper, we set this programme of investigation in motion. We have two main themes. First we introduce axets which axiomatise the notion of a set of axes together with the Miyamoto group. This allows us to formalise shapes and talk about their completions in a rigorous way. Secondly, we generalise some small collapsing shapes found in \cite{3gen4trans} specifically for algebras of type $\cM(\frac{1}{4}, \frac{1}{32})$ to whole families of shapes for arbitrary algebras of Monster type $\cM(\al, \bt)$. Furthermore, we classify all the completions, showing that these are non-trivial only in very specific exceptional situations. Our result is the first ever classification of a family of axial algebras of Monster type $\cM(\al,\bt)$ outside the $2$-generated case.

As already mentioned, axets abstract the notion of a set of axes. In an axial algebra, \emph{axes} are semisimple idempotents whose eigenvectors multiply according to a so-called \emph{fusion law}. Whenever this fusion law is $T$-graded by a group $T$, we have a natural map $\tau\colon X \times T^* \to \Aut(A)$ (the Miyamoto map) which associates to each axis $a \in X$ a subgroup of automorphisms $T_a = \{\tau_a(\chi) : \chi \in T^* \}$.  (Here $T^*$ is the group of linear characters of $T$.) The Miyamoto group $\Miy(X)$ is generated by all the $T_a$, $a\in X$. In the algebra, automorphisms take axes to axes and so the Miyamoto group $\Miy(X)$ acts on $X$.

We introduce an \emph{$S$-axet} as a $G$-set $X$, together with a Miyamoto map $\tau\colon X \times S \to G$, for some abelian group $S$, which satisfies analogous properties to that of a set of axes.  The Miyamoto group of the axet is the subgroup of $G$ generated by the image of $\tau$.  So every axial algebra $A$ involves an axet $X(A)$ given by the action of $\Miy(X)$ on the set of axes $X$, but now we can talk about an axet without an ambient axial algebra.

In addition to introducing axets, we also introduce \emph{morphisms} between axets, \ie natural maps.  Thus, we define the category $S$-$\mathbf{Axet}$ of axets.  In particular, this gives us isomorphisms of axets and so we can classify them up to isomorphism.  We note that our concept of isomorphism of axets does not require an isomorphism between the corresponding groups $G$.  Essentially, two axets are isomorphic if the action of their Miyamoto groups is the same.

We also develop the basic theory for general morphisms of axets, including subaxets, factor axets and versions of the First Isomorphism Theorem and Correspondence Theorem for axets.

\medskip
We are particularly interested in axial algebras with a fusion law which is $C_2$-graded (such as the fusion law for the Griess algebra and the generalised Monster fusion law $\cM(\al, \bt)$), so here we focus on $C_2$-axets.  One example of a $C_2$-axet is $X(n)$ of size $n$ (we allow $n = \infty$) whose $G$-action is given by the action of the dihedral group $D_{2n}$ on the $n$-gon.  We identify an entirely new example $X'(3k)$ given by gluing two (dual) copies of $X(2k)$, one of which has been `folded'.  This is isomorphic to a factor axet of $X(4k)$.

Even though an axet is just a combinatorial group-theoretic object, it still carries much information about axial algebras.  In particular, we have the following.

\begin{theorem}
A $2$-generated $C_2$-axet is isomorphic to either
\begin{enumerate}
\item $X(n)$, or
\item $X'(n)$, where $n = 3k$, $k \in \mathbb{N}$.
\end{enumerate}
\end{theorem}

In $X(n)$, there is either one orbit of axes, or two orbits of axes each of length $\tfrac{n}{2}$.  However, in $X'(3k)$, there are two orbits, one of length $k$ and the other of length $2k$.  We call the first type \emph{regular} and the second \emph{skew}.  Almost all $2$-generated axial algebras currently known have a regular axet.  However, very recently and after the initial version of this article appeared on the arXiv, the first examples of algebras with a skew axet (for $k=1$ only) have been found by Turner.

\begin{question}
Do there exist $2$-generated axial algebras of Monster type $\cM(\al, \bt)$ with a skew axet $X'(3k)$ for $k>1$?  Can one classify all such algebras?
\end{question}

A $2$-generated axial algebra $A$, generated by axes $a$ and $b$, is \emph{symmetric} if there exists an involutory automorphism $\phi$ of $A$ switching $a$ and $b$.  In the symmetric case, the orbits of $a$ and $b$ under the action of the Miyamoto group have equal size and so the axet must be regular.  The symmetric $2$-generated algebras of Monster type have been classified recently by Yabe \cite{yabe}, with some cases completed by Franchi and Mainardis \cite{highwater5} and by Franchi, Mainardis and M\textsuperscript{c}Inroy in \cite{HWquo}.  By the above, a skew axet can only occur in the non-symmetric case and the classification of such $2$-generated algebras of Monster type remains an open problem.

In this paper, we are interested in classifying shapes on small axets.  So it is natural to first consider $3$-generated axets.  However, even this is too broad without assuming extra conditions (indeed, any normal set of involutions in a group leads to a $C_2$-axet and there are infinitely many groups generated by three involutions).  We explore $3$-generated axets $X$ with one orbit under its Miyamoto group being a single axis $\{a \}$ and the remainder $X-\{ a\}$ forming a second orbit.  We find two families of examples and prove the following.

\begin{proposition}
Let $X$ be a $3$-generated axet with two orbits under its Miyamoto group $\{ a \}$ and $X-\{a \}$.  Then $X$ is isomorphic to either
\begin{enumerate}
\item $X_1(1+n)$, $n \geq 3$ finite and odd, or
\item $X_2(1+n)$, where $n = 2k$ and $k$ odd.
\end{enumerate}
\end{proposition}
Of these, the first family is regular and the second is skew.

Having laid the foundations by defining axets, we can now define shapes.  Roughly speaking, for a fusion law $\cF$, a shape $\Theta$ on an axet $X$ is a coherent assignment of a $2$-generated $\cF$-axial algebra $A_Y$ to every $2$-generated subaxet $Y \subset X$. Hence shapes are similar to group amalgams.

A \emph{morphism} of shapes is a morphism of the underlying axets which induces an algebra homomorphism on the algebras $A_Y$ in a consistent way.  Thus we also get a category of shapes.  We say that an algebra $A$ is a \emph{completion} of $\Theta$ if there is a surjective morphism of shapes from $\Theta$ to the shape $\Theta(A)$ of the algebra $A$.  A shape \emph{collapses} if it does not have any completions.

Our second main theme is classifying shapes, particularly those which collapse. We are most interested in the generalised Monster fusion law $\cM(\al, \bt)$, which is $C_2$-graded.  All previous collapsing examples were obtained computationally as individual ad-hoc examples for the $\cM(\frac{1}{4},\frac{1}{32})$ fusion law (which the Griess algebra satisfies), whereas here we exhibit families of collapsing examples in the class of $\cM(\al, \bt)$-axial algebras.

We begin by considering a $3$-generated axet $X$ with a fixed axis $\{ a \}$ as described above.  Since almost all known $2$-generated algebras for $\cM(\al,\bt)$ are regular, we only consider $X_1(1+n)$, which is regular.  As $\{ a\}$ is a single orbit, $\tau_a$ acts trivially on $X$ and so $a$ has no $\bt$-part in any completion $A$.  This means we can define an additional automorphism $\sg_a$ which inverts the $\al$-eigenspace of $a$.  By studying the action of $\sg_a$ on any completion $A$, we prove a powerful reduction statement (Proposition \ref{1+nshapeprop}).  Roughly speaking, we have two cases: either $\sg_a=1$, in which case we get a direct sum of algebras, or $\sg_a \neq 1$.  In this latter case, we show that we have a $2$-generated subalgebra $B$ with axet $X(2n)$, which is symmetric and in almost all cases $B = A$.  Moreover, $a \in B$ and $\sg_a$ switches the two orbits of axes in $B$.  We call this property (J).

Combining our reduction result and the classification of symmetric $2$-generated $\cM(\al, \bt)$-axial algebras, we are left only to check which of these $2$-generated algebras have a finite axet and which have property (J).  It is known in folklore that algebras of Jordan type $\bt=\frac{1}{2}$ (which are also algebras of Monster type $(\al,\frac{1}{2})$) generically have an infinite axet, but it can be finite for some parameters.  We make this precise in Subsection \ref{sec:jordanaxet}.  Since Yabe's classification is very recent, such behaviour for other families has not been studied yet; we determine in Section \ref{sec:2genalbt} the axets arising in the symmetric algebras. This is done case by case and the resulting statements are quite technical. We mention here a consequence that is quite unexpected\footnote{In a sense, this already follows from results from \cite{splitspin} and \cite{HWquo} (see Theorem \ref{HWideals}), but it was not stated in this form.}.

\begin{theorem}
For any $n \in \N \cup \{ \infty\}$, there exists a symmetric $2$-generated $\cM(\al, \bt)$-axial algebra with axet $X(n)$.
\end{theorem}

Most of the families of algebras of Monster type $\cM(\al, \bt)$ have a fixed finite axet.  However, two of the families generically have an infinite axet, but we list all values of the parameter where the axet is finite.  The result for quotients of the Highwater algebra follows from \cite{HWquo}.  We explore other properties of these algebras in Section \ref{sec:2genalbt}, including property (J).

Writing $\cJ^2(\al)$ for the set of $2$-generated axial algebras of Jordan type $\al$, we have the following (for a more precise version see Theorem \ref{1+nshape}).

\begin{theorem}
Let $X = \la a, b, c \ra \cong X_1(1+n)$, $n \geq 3$ odd, be a $3$-generated axet, where $a$ is the fixed axis, and let $Y = \la a,b \ra$ and $Z = X - \{ a\}$.  Suppose that $\Theta$ is a shape on $X$ for $\cM(\al, \bt)$ and $A$ is a completion.  Then
\begin{enumerate}
\item If $\sg_a = 1$, then $A_Y \cong 2\B$ and $A = 1\A \oplus A_Z$.
\item If $\sg_a \neq 1$, then $A_Y \in \cJ^2(\alpha) - \{ 2\B\}$ and either
\begin{enumerate}
\item $A$ belongs to one of five families of symmetric $2$-generated algebra with axet $X(2n)$, or
\item $(\al, \bt) = (\frac{1}{2},2)$ and $A \cong \mathrm{Bar}_{0,1}(\frac{1}{2},2)$.
\end{enumerate}
\end{enumerate}
\end{theorem}

In Theorem \ref{1+nshape}, we give the full list of possible algebras for $A_Z$ in case $1$ and the five families in case $2$, as well as a description of the $3$-generated exception $\mathrm{Bar}_{0,1}(\frac{1}{2},2)$, which is a baric algebra.

The structure of the paper is as follows.  In Section \ref{sec:background}, we briefly review the background for axial algebras and their Miyamoto groups and discuss the structure of their ideals.  The first main topic of this paper, axets, is introduced in Section \ref{sec:axet}.  We focus on $C_2$-axets, classifying $2$-generated axets and $3$-generated axets with a fixed axis under the action of the Miyamoto group. In Section \ref{sec:shapes}, we introduce shapes for axets.  We review $2$-generated axial algebras of Jordan type and $\cM(\al, \bt)$ type in Section \ref{sec:2gen}.  In particular, some of the details we need about their  ideals, quotients and axets have not appeared before.  Our reduction theorem for shapes on $X_1(1+n)$ is the subject of Section \ref{sec:reduction}.  We explore $2$-generated axial algebras of Monster type $\cM(\al,\bt)$ in Section \ref{sec:2genalbt}, in particular classifying their axets and whether they have property (J).  Finally, in Section \ref{sec:forbidden}, we combine these to prove our main theorem.

\medskip

The work of the second author was partially supported by the Mathematical Center in Akademgorodok under agreement No.\ 075-15-2019-1675 with the Ministry of Science and Higher Education of the Russian Federation. 

\medskip

We would like to thank the anonymous referee for their very useful comments.


\section{Background}\label{sec:background}

\subsection{Axial algebras}

We give brief details here introducing axial algebras; for a full account see \cite{axialstructure}.

\begin{definition}
Let $\mathbb{F}$ be a field.  A \emph{fusion law} is a set $\cF \subset \mathbb{F}$ together with a symmetric binary map $\star \colon \cF \times \cF \to 2^\cF$.
\end{definition}

We represent a fusion law in a table, similar to a group multiplication table.  In this, we drop the set notation in each cell and simply list the elements of $\lm \star \mu$, for $\lm, \mu \in \cF$.

\begin{figure}[!ht]
\begin{center}
\begin{minipage}[t]{0.25\linewidth}
\renewcommand{\arraystretch}{1.5}
\begin{tabular}[t]{c||c|c|c}
 & $1$ & $0$ & $\eta$ \\
\hline\hline
$1$ & $1$ & & $\eta$ \\
\hline
$0$ & & $0$ & $\eta$ \\
\hline 
$\eta$ & $\eta$ & $\eta$ & $1,0$
\end{tabular}
\end{minipage}
\begin{minipage}[t]{0.33\linewidth}
\renewcommand{\arraystretch}{1.5}
\begin{tabular}[t]{c||c|c|c|c}
 & $1$ & $0$ & $\alpha$ & $\beta$ \\
\hline\hline
$1$ & $1$ & & $\alpha$ & $\beta$ \\
\hline
$0$ & & $0$ & $\alpha$ & $\beta$ \\
\hline 
$\alpha$ & $\alpha$ & $\alpha$ & $1,0$ & $\beta$ \\
\hline 
$\beta$ & $\beta$ & $\beta$ & $\beta$ & $1,0,\alpha$
\end{tabular}
\end{minipage}
\end{center}
\caption{Fusion laws $\cJ(\eta)$, and $\cM(\al,\bt)$.}
\label{fusion laws}
\end{figure}

Let $A$ be a commutative algebra over $\mathbb{F}$ and $a \in A$.  By $\ad_a$ we denote the adjoint map with respect to $a$ and by $A_\lambda(a)$ we denote the $\lambda$-eigenspace of $\ad_a$ (if $\lm$ is not an eigenvalue, then $A_\lm = 0$).  We will also write $A_S(a) = \bigoplus_{\lambda \in S} A_\lambda(a)$ where $S \subset \FF$.

\begin{definition}
Let $\cF$ be a fusion law.  A commutative algebra $A$ over $\mathbb{F}$ together with a distinguished subset of elements $X$ is an $\cF$-axial algebra if
\begin{enumerate}
\item $A$ is generated by $X$;
\item for each $a \in X$,
\begin{enumerate}
\item $a$ is an idempotent;
\item $a$ is semisimple, namely $A = A_\cF = \bigoplus_{\lambda \in \cF} A_\lambda(a)$;
\item for all $\lambda, \mu \in \cF$, $A_\lambda(a) A_\mu(a) \subseteq A_{\lambda \star \mu}(a)$.
\end{enumerate}
\end{enumerate}
\end{definition}

The elements of $X$ are called \emph{$\cF$-axes} and for $2(c)$ we say that $a$ satisfies the fusion law $\cF$.  Where $\cF$ is obvious from context, we drop it and just say axes and axial algebra.  Similarly, if $a$ is obvious, we write $A_\lambda$ for $A_\lambda(a)$.  Where $X$ is understood, we will simply refer to $A$ as an axial algebra with $X$ being implicit.

We are particularly interested in axial algebras with the fusion laws in Figure \ref{fusion laws}.  We say an axial algebra is of \emph{Jordan type $\eta$} if it has the Jordan fusion law $\cJ(\eta)$, for $\eta \neq 1,0$, and is of \emph{Monster type $(\al,\bt)$} if it has the Monster fusion law $\cM(\al, \bt)$, for $\alpha, \beta \neq 1,0$, $\al \neq \bt$.  Similarly, we say that an axis is of \emph{Jordan type}, or \emph{Monster type} if it obeys the Jordan, or Monster fusion law, respectively.

Note that $a \in A_1$, so we will always assume that $1 \in \cF$.

\begin{definition}
An axis $a$ is \emph{primitive} if $A_1 = \la a \ra$ is $1$-dimensional and $A$ is \emph{primitive} if all axes in $X$ are primitive.
\end{definition}

In the primitive case, we may assume that $1 \star \lm \subseteq \{ \lm\}$ and we have equality if and only if $\lm \neq 0$.  We say that the fusion law is \emph{Seress} if $0 \in \cF$ and $0 \star \lambda \subseteq \{ \lambda \}$ for all $\lambda \in \cF$.  In this case, we have the following partial associativity result.

\begin{lemma}[Seress Lemma]\textup{\cite[Proposition 3.9]{Axial1}}
If $\cF$ is Seress, then an axis $a$ associates with $A_{\{1, 0\}}(a)$.  That is, for $x \in A$ and $y \in A_{\{1, 0\}}(a)$,
\[
a(xy) = (ax)y.
\]
\end{lemma}

Axial algebras can also have a bilinear form which associates with the algebra product.

\begin{definition}
Let $A$ be an axial algebra.  A \emph{Frobenius form} is a bilinear form $(\cdot, \cdot) \colon A \times A \to \mathbb{F}$ such that, for all $a,b,c \in A$,
\[
(a,bc)=(ab,c).
\]
\end{definition}
Note that a Frobenius form is necessarily symmetric.

\subsection{Gradings and automorphisms}

\begin{definition}
Let $\cF$ be a fusion law and $T$ be a group.  A \emph{$T$-grading} is a map $f \colon \cF \to T$ such that
\[
f(\lambda)f(\mu) = f(\nu)
\]
for all $\lambda, \mu \in \cF$ and all $\nu \in \lambda \star \mu$.
\end{definition}

This definition is in the same style as the one given for decomposition algebras in \cite{DPSV}.  Note that this definition is equivalent to the one given in \cite{axialstructure}.

\begin{definition}
Let $A$ be an axial algebra with a $T$-graded fusion law $\cF$ and $T^*$ be the group of linear characters of $T$ over $\mathbb{F}$.  The \emph{Miyamoto map} $\tau \colon X \times T^* \to \Aut(A)$ sends the pair $(a, \chi)$ to the automorphism $\tau_a(\chi) \colon A \to A$ given by
\[
v \mapsto \chi(f(\lambda))v
\]
for each $\lambda \in \cF$ and $v \in A_\lambda(a)$.
\end{definition}

We call $\tau_a(\chi)$ a \emph{Miyamoto automorphism} and $T_a := \la \tau_a(\chi) : \chi \in T^* \ra$ the \emph{axial subgroup of $a$}.  When $\cF$ is $C_2$-graded and $\mathbb{F}$ does not have characteristic $2$, there is only one non-trivial character, $\chi_{-1}$, and we write $\tau_a = \tau_a(\chi_{-1})$.

\begin{definition}
The \emph{Miyamoto group} is the group
\[
\Miy(X) := \la T_a : a \in X \ra \leq \Aut(A).
\]
\end{definition}

An automorphism $g \in \Aut(A)$ takes an axis $a \in X$ to another axis, possibly outside $X$.  We say that $X$ is \emph{closed} if it is closed under the action of its Miyamoto group $\Miy(X)$.  If $\bar{X} := X^{\Miy(X)}$, then $\Miy(\bar{X}) = \Miy(X)$ and $\bar{X}$ is closed.  Whenever we have a graded fusion law, we will typically assume that $X$ is closed.  We will also consider proper subsets $Y \subset X$, using the notation $\Miy(Y)$ and $\bar{Y} = Y^{\Miy(Y)}$ and saying that $Y$ is closed if $Y = \bar{Y}$.

\subsection{Ideals}

The structure of ideals of axial algebras was studied in \cite{axialstructure}.

First, observe that an ideal $I$ is invariant under multiplication by an axis $a$, so the axis must be semisimple when restricted to $I$.  That is, $I$ is a direct sum of eigenspaces for $\ad_a|I$.

\begin{lemma}\textup{\cite[Corollary 3.11]{axialstructure}}\label{Miyideal}
Every ideal of an axial algebra $A$ is invariant under the action of the Miyamoto group $\Miy(X)$.
\end{lemma}

Ideals in an axial algebra can be understood by naturally splitting them into two classes: those which contain an axis and those which do not contain any axes.

\begin{definition}
The \emph{radical} $R(A,X)$ of an axial algebra $A = \lla X \rra$ is the unique largest ideal of $A$ containing no axes.
\end{definition}

For axial algebras with a Frobenius form, the following allows us to easily identify the radical.

\begin{theorem}\textup{\cite[Theorem 4.9]{axialstructure}}\label{radical}
Let $A$ be a primitive axial algebra with a Frobenius form.  Then the radical $A^\perp$ of the Frobenius form coincides with the radical $R(A, X)$ if and only if $(a,a) \neq 0$ for all $a \in X$.
\end{theorem}

We now turn to ideals which do contain an axis.  Let $a, b \in X$.  Since $b$ is semisimple, we may decompose $a = \sum_{\lambda \in \cF} a_\lambda$, where $a_\lm \in A_\lm(b)$.  If $A$ is primitive, then $a_1$ is a scalar multiple of $b$ and we call $a_1$ the \emph{projection} of $a$ onto $b$.

\begin{definition}
Let $A$ be a primitive axial algebra.  The \emph{projection graph} $\Gamma$ is a directed graph with vertex set $X$ and a directed edge from $a$ to $b$ if the projection $a_1$ of $a$ onto $b$ is non-zero.
\end{definition}

\begin{lemma}\textup{\cite[Lemma 4.14]{axialstructure}}\label{idealprojgraph}
Let $A$ be a primitive axial algebra with projection graph $\Gamma$ and $I \unlhd A$.  Suppose that $a \in I$.  If there exists a directed path from $a$ to another axis $b$, then $b \in I$.
\end{lemma}

\begin{definition}\textup{\cite[Lemma 4.17]{axialstructure}}\label{projundirected}
Let $A$ be a primitive axial algebra with a Frobenius form such that $(a,a) \neq 0$ for all $a \in X$.  Then the projection graph $\Gamma$ is an undirected graph with an edge from $a$ to $b$ if and only if $(a,b) \neq 0$.
\end{definition}

In light of the above, we see that for an axial algebra with a Frobenius form which is non-zero on the axes, the ideals of $A$ containing axes are given by sums of connected components of the projection graph.  In particular, if the projection graph is connected, then there does not exist a proper ideal which contains an axis.


\section{Axets}\label{sec:axet}

The concept of an axet is new and it formalises a closed set of axes together with its Miyamoto group.  Some of the ideas in this section appeared implicitly in \cite{axialconstruction}.

\subsection{Axets}

\begin{definition}\label{axet}
Let $S$ be a group.  An \emph{$S$-axet} $(G,X,\tau)$ consists of a group $G$ and a $G$-set $X$ (i.e.\ $G$ acts on $X$)
together with a \emph{Miyamoto map} $\tau\colon X\times S\to G$ (we write $\tau_x(s)$ 
for $\tau(x,s)$) satisfying, for all $x\in X$, $s, s'\in S$, and $g\in G$: 
\begin{enumerate}
\item $\tau_x(s)\in G_x$;
\item $\tau_x(ss')=\tau_x(s)\tau_x(s')$; 
and
\item $\tau_{xg}(s)=\tau_x(s)^g$.
\end{enumerate}
\end{definition}

By an abuse of terminology, we may talk about the axet $X$, when $S$, $G$ and 
$\tau$ are assumed or clear from context. 

We call the elements of $X$ \emph{axes}. The first two properties above 
mean that $\tau_x$ is a homomorphism from $S$ to the stabiliser $G_x$ 
of the axis $x$. The third property ties the action on $X$ to conjugation 
in $G$.  More precisely it says that, for a fixed $s \in S$, the map $x \mapsto \tau_x(s)$ is $G$-equivariant.

\begin{example}\label{axialaxet}
Let $A$ be an axial algebra for a $T$-graded fusion law $\cF$ and consider 
a closed subset of axes $Y \subseteq X$.  Define $G$ to be the setwise 
stabiliser of $Y$ in $\Aut(A)$.  Then, $(G, Y, \tau)$ is a $T^*$-axet.  
When $Y = X$, we call it the \emph{axet of $A$} and write $X = X(A)$.
\end{example}

The following examples show that axets are plentiful.

\begin{example}\label{ex:axet}
Let $S = \la s \ra \cong C_2$.  
\begin{enumerate}
\item Suppose $G$ is a group and let $X$ be a normal set of involutions 
in $G$.  We define $\tau \colon X \times S \to G$ by $\tau_x(1) = 1$ and 
$\tau_x(s) = x$, for all $x \in X$.  Then $X$ is a $C_2$-axet.

\item Suppose $\Phi$ is a root system and let $G = W$ be its Weyl group.  Let $X = \{ \{ \al, -\al \} : \al \in \Phi \}$ and define $\tau_x(s)$ to be the reflection in the root $\al$, where $x = \{ \al, -\al\}$ and $\tau_x(1) = 1$.  Then $X$ is a $C_2$-axet.  In fact this example is just a more combinatorial description of the set of reflections $X$ in the Weyl group $W$.
\end{enumerate}
\end{example}

\begin{lemma} \label{axet basics}
Let $x,y\in X$.
\begin{enumerate}
\item If $x$ and $y$ are in the same $G$-orbit, then the homomorphisms $\tau_x$ and $\tau_y$ have a common kernel.
\item The image $\im(\tau_x)$ is in the centre of $G_x$.
\end{enumerate}
\end{lemma}

\begin{proof}
For the first claim, let $g\in G$ be such that $xg=y$. If 
$s\in\ker(\tau_x)$ then $\tau_y(s)=\tau_{xg}(s)=\tau_x(s)^g=1^g=1$. So 
$\ker(\tau_x)\leq\ker(\tau_y)$ and, by symmetry, we have equality.

Turning to the second claim, if $g\in G_x$ and $s\in S$ then 
$\tau_x(s)^g=\tau_{xg}(s)=\tau_x(s)$. So $\tau_x(s)\in Z(G_x)$ for each 
$s\in S$.
\end{proof}

As a consequence of the second claim, the derived subgroup $[S,S]$ is 
in the kernel of $\tau_x$ for all $x\in X$. Because of this, we may 
assume that $S$ is abelian. In the context of axial algebras, $S = T^*$ is 
the group of linear characters of the grading group $T$ of the fusion law. 
So $S$ is automatically abelian.

\begin{definition}
Suppose $(G,X,\tau)$ is an $S$-axet. 
\begin{enumerate}
\item The subgroup $T_x:=\im(\tau_x)\leq G$ is the \emph{axial 
subgroup} corresponding to the axis $x$.
\item The \emph{Miyamoto group} of the axet $X$ is the subgroup
$$\Miy(X)=\la T_x:x\in X\ra \leq G.$$
\end{enumerate}
\end{definition}

For the example of an axial algebra $A$ in Example \ref{axialaxet}, the Miyamoto group of its axet $X$ is the Miyamoto group of the axial algebra.  Note that, just as the Miyamoto group of an axial algebra is not necessarily its full automorphism group, the Miyamoto group of an axet $X$ can be a proper subgroup of $G$.

\begin{lemma}
\begin{enumerate}
\item If $g\in G$ and $x\in X$ then $T_x^g=T_{xg}$.
\item The Miyamoto group $\Miy(X)$ is a normal subgroup of $G$.
\end{enumerate}
\end{lemma}

\begin{proof}
The first claim follows from the last axiom of axets, and the second 
claim follows from the first one.
\end{proof}

Note that $G$ is not assumed to act faithfully on $X$. We call an axet 
\emph{faithful} when the $G$-set in it is faithful.

\begin{lemma}\label{Miycentre}
Let $\pi\colon G\to\Sym(X)$ be the action of $G$ on $X$.  Then $K = \ker(\pi)$ centralises $\Miy(X)$. In 
particular, the kernel $\Miy(X)\cap K$ of $\Miy(X)$ acting on $X$ is in 
the centre of $\Miy(X)$.
\end{lemma}

\begin{proof}
If $g\in K$ then $g$ is contained in each axis stabilizer $G_x$, which 
means, by Lemma \ref{axet basics}, that $g$ centralises each axial 
subgroup $T_x$.
\end{proof}

So if we define $\bar G=\im(\pi)$, $\bar X :=X$ and $\bar\tau = \pi \circ \tau$, then $(\bar{G}, \bar{X}, \bar{\tau})$ is an $S$-axet which is faithful.  However, the group $\bar{G}$ may still have some automorphisms which are not in $\Miy(\bar{X}) = \overline{\Miy(X)}$, so we want to define an even smaller axet.

\begin{definition}
Suppose $(G,X,\tau)$ is an $S$-axet and let $\pi\colon G\to\Sym(X)$ be the action of $G$ on $X$. Then
$(\overline{\Miy(X)},\bar X,\bar\tau)$ is an $S$-axet, which we call the 
\emph{core} of the axet $(G,X,\tau)$.
\end{definition}

We define the \emph{action Miyamoto group} $\AMiy(X) := \overline{\Miy(X)}$.  Clearly, $\AMiy(X) \cong 
\Miy(X)/(\Miy(X) \cap K)$, where, as above, $K$ is the kernel of the action $\pi$ of $G$ on $X$.  In other words, $\Miy(X)$ is an extension of $\AMiy(X)$ and according to Lemma \ref{Miycentre}, it is in fact a central extension.

We now turn to subaxets.

\begin{definition}
Suppose $(G,X,\tau)$ is an $S$-axet. A subset $Y\subseteq X$ is 
\emph{closed} if $Y$ is invariant under each axial subgroup $T_y$, 
$y\in Y$.
\end{definition}

It is easy to see that the intersection of closed subsets is itself 
closed. This allows us to introduce, for a subset $Y\subseteq X$, its 
\emph{closure} $\la Y\ra$ as the smallest closed subset containing $Y$.  (We note that $Y$ is closed 
if and only if $Y = \la Y \ra$.)  We will also say that $Y$ \emph{generates} $\la Y \ra$ and that 
$\la Y \ra$ is $k$-generated if it is generated by a set of axes $Y$ of cardinality $k$.

\begin{definition}\label{subaxet}
A \emph{subaxet} $(H,Y,\sg)$ of the axet $(G,X,\tau)$ consists of a 
closed subset $Y\subseteq X$, its set-wise stabiliser $H=G_Y$, and the 
map $\sg$, which is the restriction of $\tau$ to the set $Y\times S$.
\end{definition}

Note that since $Y$ is closed, the axial subgroup $T_y$ is contained in 
$H$ for each $y\in Y$. So $(H,Y,\sg)$ is indeed an $S$-axet.  Correspondingly, we can talk about $\Miy(Y)$, 
$\AMiy(Y)$ and the core of $Y$.  For the example of an axial algebra $A$ in Example \ref{axialaxet}, if $B$ is a sub axial algebra with a closed set of axes $Y$, then $(H, Y, \sg)$ is indeed a subaxet of $X$.

Our aim is to define the category of axets. One way is to define a morphism of $S$-axets $(G,X,\tau)$ and $(G',X',\tau')$ 
as a pair of maps $(\phi,\psi)$, where $\phi:G\to G'$ is a homomorphism and $\psi:X\to X'$ is a map, such that the natural consistency 
conditions on the action and Miyamoto map are satisfied. However, as it turns out, this concept is too restrictive and often 
a natural map $\psi$  between the two sets of axes has no corresponding group homomorphism $\phi$. Because of this, we opt for a weaker concept that better suits 
our needs.

\begin{definition}
For $S$-axets $(G, X, \tau)$ and $(G',X', \tau')$, a \emph{morphism} from $X$ to $X'$ is a map of sets $\psi \colon X \to X'$ 
such that for all $x,y\in X$ and $s\in S$,
$$
\psi(y\tau_x(s))=\psi(y)\tau'_{\psi(x)}(s).
$$
\end{definition}

It is immediate that the composition of two morphisms is again a morphism.  So we indeed have a category of $S$-axets, 
denoted $S$-$\mathbf{Axet}$.  In the remainder of this subsection, we investigate this category, specifically the morphisms, in more detail.

Firstly, let us identify the isomorphisms, \ie invertible morphisms.

\begin{proposition} 
If a morphism $\psi:X\to X'$ is bijective then $\psi^{-1}:X'\to X$ is also 
a morphism of axets.
\end{proposition}

\begin{proof}
We need to show that, for all $x',y'\in X'$ and all $s\in S$, we have that 
$\psi^{-1}((y')\tau'_{x'}(s))=\psi^{-1}(y')\tau_{\psi^{-1}(x')}(s)$. Let 
$x=\psi^{-1}(x')$ and $y=\psi^{-1}(y')$ (that is, $x'=\psi(x)$ and $y'=\psi(y)$). 
Since $\psi$ is a morphism, we have $\psi(y)\tau'_{\psi(x)}(s)=\psi(y\tau_x(s))$. 
Applying $\psi^{-1}$ to both sides of this equality, we obtain 
$\psi^{-1}(\psi(y)\tau'_{\psi(x)}(s))=y\tau_x(s)$, which is exactly the condition 
we require.
\end{proof}

Thus, isomorphisms are just bijective morphisms.  As usual, we try to classify axets up to isomorphism. In this context, let us mention 
the following fact.

\begin{proposition} \label{core isomorphic}
Every axet is isomorphic to its core.
\end{proposition}

The required isomorphism is simply the identity map.  This proposition justifies our introduction of the core axet, which serves as a canonical version of the axet.

Our next task is to develop the structure theory of axet morphisms similar to the 
First Isomorphism Theorem for groups. We begin by discussing congruences on axets.

\begin{definition}
Suppose $(G,X,\tau)$ is an $S$-axet. A \emph{congruence} $\sim$ on $X$ is an equivalence 
relation on $X$ such that if $x\sim y$ and $z\sim w$ then $x\tau_z(s)\sim y\tau_w(s)$ 
for $x,y,z,w\in X$ and all $s\in S$. 
\end{definition}

We will use $[x]$ for the congruence class of $x\in X$.  Note the following properties of congruences. 

\begin{lemma}\label{axetcongruence}
Suppose $\sim$ is a congruence on an $S$-axet $(G,X,\tau)$. Then
\begin{enumerate}
\item $\sim$ is invariant under the action of $\Miy(X)$, \ie $x\sim y$ if and only if 
$xg\sim yg$ for $x,y\in X$ and $g\in\Miy(X)$; consequently, $\Miy(X)$ acts on the set of congruence classes $X/\!\sim$ via $[x]g = [xg]$;
\item if $z\sim w$ then the action of $\tau_z(s)$ and $\tau_w(s)$ 
on the set of congruence classes $X/\!\sim$ is the same.
\end{enumerate}
\end{lemma}

\begin{proof}
It suffices to show that $\sim$ is invariant under $g=\tau_z(s)$ for all $z\in X$ 
and $s\in S$. Suppose $x\sim y$ for some $x,y\in X$ and set $x'=x\tau_z(s)$ and $y'=y\tau_z(s)$. 
The congruence condition taken with $w=z$ gives us that $x\tau_z(s)\sim y\tau_z(s)$, \ie 
$x'\sim y'$. Conversely, suppose that $x'\sim y'$. We again use the congruence condition, this time with $\tau_z(s^{-1}) = \tau_z(s)^{-1}$ to get that $x = x'\tau_z(s^{-1}) \sim y'\tau_z(s^{-1}) = y$.

For the second part, let $z \sim w$.  Since $x \sim x$, we can write $x\tau_z(s) \sim  x\tau_w(s)$.  Therefore, $[x]\tau_z(s) = [x]\tau_w(s)$ for all $s \in S$.
\end{proof}

These properties show that the following concept is well-defined.

\begin{definition}
The \emph{factor axet} $X/\!\sim$ is the axet $(\bar G,\bar X,\bar\tau)$, where $\bar X=X/\!\sim$ is 
the set of congruence classes of $\sim$, $\bar G$ is the group induced by $\Miy(X)$ acting on $\bar X$, 
and $\bar\tau_{[x]}(s)\in\bar G$ is the permutation of $\bar X$ induced by $\tau_x(s)$.
\end{definition}

Note that the factor axet immediately appears in its canonical core version as the group $G$ here is the action Miyamoto group.  We could in principle have defined the group part of the axet differently, taking a larger group.  However, in view of Proposition \ref{core isomorphic}, any other version of the factor axet would be isomorphic to the core version given above.

\begin{theorem}[First Isomorphism Theorem]
Suppose that we have a morphism of axets $\psi \colon X\to X'$.  Then
\begin{enumerate}
\item the image $\psi(X)$ is a subaxet of $X'$;
\item the equivalence $x \sim y$ given by $\psi(x) = \psi(y)$ is a congruence on $X$; and
\item $\psi(X) \cong X/\!\sim$.
\end{enumerate}
\end{theorem}
\begin{proof}
Let $x', y' \in \psi(X)$.  Say, $x' = \psi(x)$ and $y' = \psi(y)$ for some $x,y \in X$.  Then, $y'\tau'_{x'}(s) = \psi(y)\tau'_{\psi(x)}(s) = \psi(y\tau_x(s))$, which is in $\psi(X)$.  So, $\psi(X)$ is closed in $X'$ and hence it is a subaxet, as claimed.

Let $x \sim y$ and $z \sim w$.  Then $\psi(x\tau_z(s)) = \psi(x)\tau'_{\psi(z)}(s) = \psi(y)\tau'_{\psi(w)}(s) = \psi(y\tau_w(s))$ and so $x\tau_z(s) \sim y\tau_w(s)$.  Hence, $\sim$ is a congruence on $X$.

Finally, the congruence classes of $\sim$ are fibres of $\psi$ and so we have a bijection $\bar{\psi}$ between $\bar{X} = X/\!\sim$ and $\im \psi$.  We just need to check that this bijection is a morphism.  Let $[x],[y] \in \bar{X}$.  Then, $\bar{\psi}([y]\bar{\tau}_{[x]}(s)) = \bar{\psi}([y\tau_{x}(s)]) = \psi(y\tau_{x}(s))= \psi(y)\tau'_{\psi(x)}(s) = \bar{\psi}([y])\tau'_{\bar{\psi}([x])}(s)$ and so the bijection $\bar{\psi}$ is indeed an isomorphism.
\end{proof}

Let us now consider the groups involved.  If $\psi \colon X \to X'$ is a morphism of axets, then it is also a isomorphism between the core of $X$ and the core of the image axet $\psi(X)$.

\begin{proposition}
Suppose $(G,X, \tau)$ and $(G',X', \tau')$ are core axets and $\psi \colon X \to X'$ is a surjective morphism.  Then, there exists a unique surjective homomorphism $\phi \colon G \to G'$ such that
\[
\phi(\tau_x(s)) = \tau'_{\psi(x)}(s)
\]
for all $x \in X$, $s \in S$.
\end{proposition}
\begin{proof}
Let $\sim$ be the congruence associated with $\psi$, $\bar{X} = X/\!\sim$ be the factor axet and let $\bar{\psi} \colon \bar{X} \to X' = \psi(X)$ be the isomorphism coming from the First Isomorphism Theorem.  By Lemma \ref{axetcongruence}, the group $G = \AMiy(X)$ has a natural action on $\bar{X}$ given by $[y]g = [yg]$.  Using the bijection $\bar{\psi}$, we translate this action to $X'$ to get $y'g = \bar{\psi}([y]g) = \bar{\psi}([yg])$, where $[y] = \bar{\psi}^{-1}(y')$, \ie $y'= \psi(y)$.  Hence, $y'g = \psi(yg)$.  We let $\phi \colon G \to \Sym(X')$ be the action homomorphism,

We claim that $\phi(G) = G'$ and the condition in the proposition is satisfied.  We check the action of $\phi(\tau_x(s))$ on $y'$:
\[
\phi(\tau_x(s))(y') = y' \tau_x(s) = \psi(y\tau_x(s)) = \psi(y) \tau'_{\psi(x)}(s) = y' \tau'_{\psi(x)}(s)
\]
and hence $\phi(\tau_x(s)) = \tau'_{\psi(x)}(s)$ as claimed.  Since $G' = \AMiy(X')$, we immediately have that $\phi(G) = G'$.
\end{proof}

As we can now see, there is in fact a correspondence between groups involved in an axet morphism.  But this correspondence is a group homomorphism only when we deal with the action Miyamoto groups (the Miyamoto groups of the cores).  In other situations, $G$ is a central extension of the Miyamoto group, possibly with some extra automorphisms on top.  So this is why we cannot expect a group homomorphism to accompany a general axet morphism: the correspondence we get is more like an isogeny.

Let us now mention some additional properties of axet morphisms which lead to a correspondence theory for axets.

\begin{proposition}
Suppose that $\psi \colon X\to X'$ is a morphism of axets. 
\begin{enumerate}
\item For $Z\subseteq X$, we have that $\psi(\la Z\ra)=\la\psi(Z)\ra$; i.e., $\psi$ takes closed sets to 
closed sets.
\item For $Z'\subseteq X'$, we have $\psi^{-1}(\la Z'\ra)=\la\psi^{-1}(Z)\ra$; i.e., full preimages 
under $\psi$ of closed sets are closed.
\end{enumerate}
\end{proposition}

The proof of this is straightforward and in a large part repeats the previous arguments.  As a corollary of (2), we have that every congruence class $[x]$ is a subaxet of $X$ as it is the preimage of the single point axet $\{\psi(x)\}$ in $X'$.

Let $\psi\colon X \to X'$ be a morphism and $\sim$ be the corresponding congruence.  We say that a subaxet $Y$ of $X$ is \emph{complete} with respect to the morphism $\psi$ if, for every $y \in Y$, its full congruence class $[y]$ with respect to $\sim$ is contained in $Y$.  Note that a complete subaxet is just a union of congruence classes which form a subaxet in the factor axet.  This immediately gives us the following.

\begin{theorem}[Correspondence Theorem]
Suppose that $\psi \colon X\to X'$ is a morphism of axets. There is a bijection between the set of complete subaxets of $X$ and the subaxets of $\im \psi$.
\end{theorem}

This is clearly very similar to the Correspondence Theorem for group homomorphisms.  The later Theorem also contains a statement about the correspondence for normal subgroups.  This can again be generalised for axets.  We say that a congruence $ \approx$ on $X$ is a \emph{coarsening} of $\sim$ if it is a coarsening of $\sim$ as an equivalence relation, \ie if $x \sim y$, then $x \approx y$, for all $x, y \in X$.  Then, the `normal correspondence' in the case of axets states that the coarsenings of $\sim$ are in a natural bijection with the congruences of $\im \psi$.

\medskip

For the final topic of this subsection, let us discuss the orbits of $G$ on $X$. Suppose that $Y$ is one 
of the orbits. Then $Y$ is closed and $G_Y=G$, so $(G,Y,\tau|_{Y\times S})$ is a subaxet. In this way 
$X$ decomposes as a union of disjoint subaxets corresponding to the orbits of $G$ on $X$. Let us 
formalise this union operation as follows.

\begin{definition}
Suppose $(G,X_i,\tau_i)$, $i\in I$, are axets with the same group $G$. The \emph{union} of these axets 
is the axet $(G,X,\tau)$, where $X$ is the disjoint union of all $X_i$ and $\tau$, similarly, is the union of all maps $\tau_i$.
\end{definition}

Looking at Definition \ref{axet}, it is easy to see that the union $X$ is indeed an axet.


\subsection{$2$-generated $C_2$-axets}

We are most interested in $C_2$-axets and a natural place to start is to consider $2$-generated axets.  Since $S = C_2$, we will write $\tau_x$ for $\tau_x(s)$, where $s$ is the generator of $S$, and say that $\tau$ is a map from $X$ to $G$.

Suppose that $X = \la a,b\ra$ is a $2$-generated $C_2$-axet, where $a \neq b$.  Then, $\tau_a$ and $\tau_b$ have order at most $2$ and hence $\Miy(X)$ is either trivial, $C_2$, or a dihedral group $D_{2n}$, or $D_\infty$.

\begin{example}
\begin{enumerate}
\item For a finite $n \geq 1$, consider the action of $G = D_{2n}$ on the regular $n$-gon.  Let $X = X(n)$ be the set of all vertices of the $n$-gon; this has size $n$.  We turn $(G, X)$ into a $C_2$-axet by defining $\tau_x$ to be the generator of the stabiliser $G_x \cong C_2$ for each $x \in X$.  We will call $X$ the \emph{$n$-gonal axet}.
\item For the infinite case, $G = D_\infty$ acts transitively on $X := X(\infty) = \mathbb{Z}$.  Set $\tau_x$ to be the reflection in $x$ for each $x \in X$.  We will call $X$ the \emph{$\infty$-gonal axet}, or the \emph{apeirogon axet}.
\end{enumerate}
\end{example}

From now on, the notation $X(n)$ includes the case of $n = \infty$.  Note that $X(n)$ is faithful provided $n \geq 3$.  Let us record a few further facts about the action of the Miyamoto group on the axet $X(n)$.

\begin{lemma}\label{Xorbit}
Let $X = X(n)$.
\begin{enumerate}
\item If $n$ is finite and odd, then $\Miy(X)  = G = D_{2n}$ is transitive on $X$.
\item If $n$ is finite and even, then $\Miy(X) \cong D_{n}$ has index two in $G = D_{2n}$ and has two orbits of equal length $\frac{n}{2}$ on $X$.   If $X = \la a, b\ra$ for axes $a, b \in X$, then $a$ and $b$ are in different orbits.
\item If $n = \infty$, then $\Miy(X) \cong D_\infty$ has index two in $G = D_\infty$ and it has two infinite orbits on $X$.  Again, if $X = \la a, b\ra$, then $a$ and $b$ are in different orbits.\footnote{So in this context, infinity appears to be even!}
\end{enumerate}
\end{lemma}

We can consider the above examples to be \emph{vertex axets}.  In a similar way to above, we can define the \emph{edge axet} $X^\ast(n)$ which has as axes the edges of the $n$-gon with the natural action of $D_{2n}$ and for an edge $e$, $\tau^\ast_e$ being the generator of the stabiliser of $e$.  However, this is not a new axet as the edge axet is isomorphic (in a sense, dual) to the vertex axet.

If $n$ is even or infinite, then the involutions in the image of $\tau^\ast$ form a different (dual) class in $G$ to those in the image of $\tau$, and so the isomorphism between $X^\ast(n)$ and $X(n)$ corresponds to an outer automorphism of $G$. If $n$ is odd then $X^\ast(n)$ is isomorphic to $X(n)$ in a more direct way.  Namely, the isomorphism $\psi$ sends every edge to the opposite vertex and this corresponds to the identity automorphism of $G$.

We could have also defined an axet using vertices and edges, but in fact this does not give us a new axet. Indeed, what we get is the union of $X(n)$ and its dual $X^\ast(n)$, and this union is isomorphic to the $2n$-gonal vertex axet $X(2n)$.

Note that $G$ acts transitively on $X = X(n)$ and recall that it acts faithfully if and only if $n \geq 3$.  Also, by Lemma \ref{Xorbit}, $\Miy(X) = G$ if and only if $n$ is odd.  When $n$ is even, or infinite, then $\Miy(X) \cong D_n$ is of index $2$ in $G$.

\begin{lemma}[Folding Lemma]
Let $(G, X, \tau)$ be an $S$-axet and suppose that $\sim$ is a congruence on $X$ with the additional condition that if $x, x' \in X$ with $x \sim x'$, then $\tau_x(s) = \tau_{x'}(s)$ for all $s \in S$.  Let $\tilde{X} = X/\!\sim$ be the set of equivalence classes and define $\tilde \tau \colon \tilde X \times S \to G$ naturally.  Then $(G, \tilde X, \tilde \tau)$ is an $S$-axet and $\Miy(\tilde X) = \Miy(X)$.
\end{lemma}

Clearly, $\tilde X$ is isomorphic to the factor axet $X/\!\sim$.  The map from this lemma will be referred to as the \emph{folding map} and it is a surjective morphism from $X$ to $\tilde X$.

Note that although we have the same group $G$ for $X$ and $\tilde X$, the kernel of the action may have increased by folding.  Since the action of $G$ on $\tilde{X}$ may not be faithful, $\tilde{X}$ is not the factor axet, but rather is isomorphic to it.  Note that the folding map, being a morphism, preserves closure and generation.

When $n$ is finite and even, we have a folding congruence on $X(n)$.  Indeed, we can make opposite vertices in the $n$-gon equivalent.  However, this folding does not give us a new example as $\tilde X(n)$ is isomorphic to $X(\frac{n}{2})$.  Yet, we can still use the folding lemma to give us a new example.

\begin{example}
Let $n = 3k$ be finite and consider $X = X(4k)$ formed from the $4k$-gon.  As above, $G = D_{8k}$, but $\Miy(X) = D_{4k}$, a subgroup of index $2$.  So, we may consider the axet $(\Miy(X), X, \tau)$ which now has two orbits under the action of $\Miy(X) = D_{4k}$, each of length $2k$.  Note that on either orbit we have a folding congruence given by pairing opposite vertices in the $4k$-gon.  Let $X' := X'(3k) = X'(n)$ be the set formed by folding one of the orbits of axes.  Then, $(D_{4k}, X', \tau')$ is an axet with $n = 3k$ axes.  Moreover, $\Miy(X') = \Miy(X) = D_{4k}$ and this action is faithful if $k \neq 1$.  We call $X'$ a \emph{skew $2$-generated axet}.
\end{example}

Let us again record the basic properties of this new axet.

\begin{lemma}\label{X'orbit}
If $X = X'(3k)$, then $\Miy(X) \cong D_{4k}$ and it has orbits of length $2k$ and $k$ on $X$.  Furthermore, if $X = \la a, b\ra$, then $a$ and $b$ are in different orbits.
\end{lemma}

Note that if we fold the other orbit instead, this results in an isomorphic (dual) version of the same axet, which is conjugate to $X'(n)$ via an element of $G = D_{8k}$.  It is also clear that $X'(n)$ is not isomorphic to $X(n)$ as the orbit structure under the Miyamoto group is different.

We may now classify the $2$-generated faithful $C_2$-axets.

\begin{theorem}\label{2genaxet}
Let $X = \la a,b\ra$ be a $2$-generated $C_2$-axet with $n$ axes, where $n \geq 2$, or $n = \infty$.  Then $X$ is isomorphic to either
\begin{enumerate}
\item $X(n)$, or
\item $X'(n)$, where $n = 3k$ for $ k \in \mathbb{N}$.
\end{enumerate}
\end{theorem}

\begin{proof}
Without loss of generality, by Proposition \ref{core isomorphic}, we may assume that $X$ is equal to its core and so $G = \Miy(X) = \AMiy(X)$ is faithful on $X$.  Since $X$ is $2$-generated and $\tau_a$ and $\tau_b$ have order at most $2$, $G$ is either trivial, $C_2$, a dihedral group $D_{2m}$, or $D_\infty$.  Observe that, by axiom (3) for axets, $\tau$ is a morphism from the $G$-set $X$ to the $G$-set of involutions from $G$ (see Example \ref{ex:axet}).  Also, by definition, $X = a^{G} \cup b^{G}$.

We start with the finite case.  Suppose first that $G \cong D_{2m}$, with $m \geq 2$ and so $\tau_a$ and $\tau_b$ both have order $2$.  Since $\tau_a \in G_a$, by the orbit-stabiliser theorem, $a^{G}$ has size at most $m$.  Similarly for $b$.  

If $m$ is odd, then there is one conjugacy class of involution in $D_{2m}$ of size $m$.  By axiom (3), $|a^G| \geq |\{ \tau_{a^g} : g \in G \}| = |\{\tau_a^g : g \in G \}| = |\tau_a^G| = m$.
So, $|a^{G}| = m$ and, similarly, $|b^{G}| = m$. Furthermore, these orbits are either equal, or disjoint.  Let $Y$ be one of these orbits.  Then $\tau$ is an isomorphism of axets from $Y$ to the axet whose set is the set of involutions $\tau(Y)$.  Note that our example $X(m)$ is also isomorphic to the axet of involutions in $D_{2m}$ when $m$ is odd.  Hence, if $a^{G} = b^{G}$, then $X$ is isomorphic to $X(m)$.  On the other hand, if $a^{G}$ and $b^{G}$ are disjoint, then each is isomorphic to $X(m)$ and so $X$ is their union. Since $m$ is odd, $X(m)$ is the same as its dual and so $X=X(m)\cup X^\ast(m)$, which is isomorphic to $X(2m)$.

If $m$ is even, then the generating involutions $\tau_a$ and $\tau_b$ are in different conjugacy classes of $G$, each of size $\frac{m}{2}$.  By axiom (3), $a^{G}$ and $b^{G}$ must also be disjoint and can have size $m$, or $\frac{m}{2}$.  Suppose first that $X$ has two orbits of length $m$.  Then one orbit is isomorphic to the vertex axet $X(m)$ of the $m$-gon and the other is isomorphic to the edge axet $X^\ast(m)$ of the same $m$-gon.  So their union $X$ is isomorphic to $X(2m)$.  Now suppose that one orbit $Y$ has length $\frac{m}{2}$, then $G$ induces $D_m$ on $Y$, and so $Y$ is the folded $X(m)$, isomorphic to $X(\frac{m}{2})$.  If the other orbit $Z$ is of length $m$ then $Z=X^\ast(m)$, because $\tau_b$ is not in the same class as $\tau_a$. Thus, $X$ is the union of the folded $X(m)$ and $X^\ast(m)$, which is precisely $X'(\frac{3}{2}m)$. (Note that if $G \cong D_4=V_4$, then $X$ being faithful implies that both orbits must have length $2$.)  Finally, suppose that both orbits are of length $\frac{m}{2}$.  Then $(\tau_a\tau_b)^\frac{m}{2}$ fixes each axis and so it is in the kernel, a contradiction to $G$ being faithful.

We now deal with the remaining small cases.  If the Miyamoto group is trivial, then clearly $X = \{a, b\}$ and $X$ is isomorphic to $X(2)$.  Suppose that $G \cong C_2 = \la g \ra$ and, without loss of generality, let $\tau_a = g$.  Then, $\tau_a$ must act non-trivially on $b$ and so $b^{G} = \{ b, b'\}$.  Since $\tau_x \in G_x$, $a^{G} = \{ a\}$ and also $\tau_b = \tau_{b'} = 1$.  In this case, $X$ is isomorphic to $X'(3)$.

Finally, the infinite case is similar to the generic even case.  The involutions $\tau_a$ and $\tau_b$ are in two different $G$-conjugacy classes and hence $a^{G}$ and $b^{G}$ are two disjoint orbits.  One orbit is isomorphic to $X(\infty)$ and the other to $X^*(\infty)$.  Hence their union $X$ is isomorphic to $X(\infty)$.
\end{proof}

\begin{definition}\label{def:skew}
We will call an axet $X$ \emph{skew} if it contains a $2$-generated subaxet isomorphic to a skew axet $X'(n)$.  Otherwise, we will call $X$ \emph{regular}.
\end{definition}


\subsection{Some $3$-generated $C_2$-axets}\label{sec:1+n}

We consider some $3$-generated $C_2$-axets which we will return to later in the paper. Let us first define the \emph{trivial one point extension} of an $S$-axet $X$
as the $S$-axet on the $G$-set $Y=\{a\}\cup X$, where $a\notin X$ and $G$ fixes $a$ and acts on $X$ as before. Similarly, we extend the Miyamoto map from $X\times S$ to $Y\times S$ 
by setting $\tau_a(s)=1$ for all $s\in S$. Note that $\Miy(Y)\cong\Miy(X)$ and $Y$ is faithful if and only if $X$ is faithful.

\begin{example}
\begin{enumerate}
\item Suppose $n$ is arbitrary (including $\infty$).  Let $X_1 = X_1(1+n) = \{ a\} \cup X(n)$ be the \emph{trivial one point extension} of $X = X(n)$.
Since $X$ is faithful when $n\geq 3$, so is $X_1$.

\item Let $n=2k$ be finite and even.  Again consider a one point extension $X_2 = X_2(1+n) = \{a\} \cup X(n)$ but this time define $\tau_a \in D_{2n}$ to be the rotation of the $n$-gon through $180^\circ$ (if $n \neq 2$, then it is the unique central involution).  We call this the \emph{central one point extension} of $X = X(n)$.  Recall that $\Miy(X(n)) = D_n$ and there are two $\Miy(X)$-orbits.  Now $\Miy(X_2) = G = D_{2n}$ if and only if $k$ is odd (if $k$ is even, then $\Miy(X_2) = D_n$ is of index $2$).  Again $X_2$ is faithful if and only if $X$ is faithful, i.e.\ for $n \neq 2$.
\end{enumerate}
\end{example}

The subaxets of $X_1(1+n)$ are $Y$ and $Y \cup \{a\}$, where $Y$ is a subaxet of $X(n)$.  In particular, $\{ a, x\} \cong X(2)$ for each $x \in X_1 - \{a \}$.  Whereas in $X_2(1+n)$, $\la a, x\ra \cong X'(3)$ for every $x \in X_2 - \{a \}$.  So $X_1(1+n)$ is regular and $X_2(1+n)$ is skew.

\begin{proposition}\label{1+n}
Let $X = \la a,b,c \ra$ be a $3$-generated axet which has two orbits under its Miyamoto group, $\{a\}$ and $Y := X - \{a\}$.  Then $X$ is isomorphic to either
\begin{enumerate}
\item $X_1(1+n)$, with $n$ finite and odd; or
\item $X_2(1+n)$, with $n = 2k$ and $k$ odd.
\end{enumerate}
In both cases, $\Miy(X) = D_{2n}$ unless $X$ is isomorphic to $X_2(1+2)$ and then $\Miy(X) = C_2$.
\end{proposition}
\begin{proof}
Without loss of generality, by Proposition \ref{core isomorphic}, we may assume that $X$ is equal to its core.  Suppose first that $\tau_a =1$.  Then, $\la \tau_b, \tau_c \ra = \Miy(X)$ and $\la b,c \ra$ is a $2$-generated subaxet with one orbit of axes.  By Theorem \ref{2genaxet} and Lemmas \ref{Xorbit} and \ref{X'orbit}, $n$ is finite and odd and $Y \cong X(n)$, hence $X = X_1(1+n)$.

Suppose now that $\tau_a \neq 1$.  By Lemma \ref{axet basics}, $\tau_a \in Z(G)$.  In particular, $\la \tau_b, \tau_c \ra$ has index at most $2$ in $\Miy(X)$ and hence $\la \tau_b, \tau_c \ra$ has either a single orbit, or two orbits of equal size on $Y$.  In the first case, by Theorem \ref{2genaxet}, $Y$ is isomorphic to $X(n)$ with $n$ odd.  However, in this case the centraliser of $\Miy(Y)$ in the symmetric group on $Y$ is trivial and hence $\tau_a = 1$, a contradiction.

Thus $Y$ has two orbits $Y_1$ and $Y_2$ under the action of $\la \tau_b, \tau_c \ra$ and $\tau_a$ fuses these two orbits.  In particular, for $y_1 \in Y_1$, $y_2 = y_1^{\tau_a} \in Y_2$ and $\tau_{y_1} = \tau_{y_2}$.  If $b$ and $c$ are from different orbits, then $\la b, c\ra = Y$ and $\im \tau|_{Y_1} = \im \tau|_{Y_2}$, so by Theorem \ref{2genaxet} and Lemmas \ref{Xorbit} and \ref{X'orbit}, $Y$ is isomorphic to $X(2k)$ with $k$ odd.  Therefore, $X$ is isomorphic to $X_2(1+n)$ with $n = 2k$ and $k$ odd.  If $c$ is in the same orbit as $b$, then we swap $c$ for $c^{\tau_a}$ and the above argument applies.
\end{proof}

Note that $X_2(1+2)$ is in fact isomorphic to the $2$-generated axet $X'(3)$.  However, by Theorem \ref{2genaxet}, all the other axets in the above proposition are $3$-generated and not $2$-generated.


\section{Shapes}\label{sec:shapes}

We wish to introduce the notion of shape on an axet and we begin with a motivating example. Let $X = X(A)$ be the axet of 
an axial algebra $A$. Recall that when $A$ has a $T$-graded fusion law then its axet $X(A)$ is an $S$-axet for $S=T^*$. 
Every $2$-generated subaxet $Y\subset X=X(A)$ generates a $2$-generated sub axial algebra $A_Y = \lla Y \rra_A$ of $A$.  
The shape $\Theta(A)$ of $A$ is the collection of the embeddings $\theta_Y  = \mathrm{id}_Y \colon Y \to A_Y$, one for each such $Y$.  
If we have an abstract axet $X$, then there is no ambient algebra $A$, but we still wish to have a similar notion of 
shape.

Let us start by introducing additional notation and terminology. First of all, in this section it is assumed throughout 
that all axial algebras have a $T$-graded fusion law $\cF$ and all axets are $S$-axets, where $S=T^*$. 

\begin{definition}
An \emph{embedding} of an axet $Y$ into an axial algebra $A$ is an injective morphism $\theta\colon Y \to X=X(A)$. We say 
that an embedding is \emph{full} if $\theta$ is surjective, \ie it is an isomorphism.
\end{definition}

Given an embedding $\theta$, we will often identify an axis $a\in X$ with its image $\theta(a)$ in the algebra; i.e., we may write 
$a$ for $\theta(a)$.

We can now introduce shapes.  For an axet $X$, we denote the set of all $1$- and $2$-generated subaxets by $\mathcal{X}_2 = \mathcal{X}_2(X)$.

\begin{definition}\label{shape}
For an axet $(G,X,\tau)$, suppose that $\Theta = \{ \theta_Y : Y \in \mathcal{X}_2 \}$, where $\theta_Y \colon Y \to X(A_Y)$ 
is a full embedding of $Y$ into an axial algebra $A_Y$ for each $Y \in \mathcal{X}_2$.  We say that $\Theta$ is a \emph{shape} on $X$ if for all 
$g \in G$ and $Y, Z \in \mathcal{X}_2$ such that $Z^g \subseteq Y$, there exists an injective algebra homomorphism $\phi_{g,Z,Y} 
\colon A_Z \to A_Y$ such that
\begin{equation}\label{consistency req}
\theta_Y \circ \psi_g = \phi_{g,Z,Y} \circ \theta_Z,
\end{equation}
where $\psi_g \colon Z \to Y$, $z \mapsto z^g$ is the morphism induced by $g$.
\end{definition}

The consistency condition (\ref{consistency req}) in this definition is illustrated in the following commutative diagram:
\[
\begin{tikzcd}[row sep = large, column sep = large]
Z \arrow[r, hook, "\psi_g"] \arrow[d, swap, "\theta_Z"] & Y \arrow[d, "\theta_Y"] \\
A_Z \arrow[r, hook, swap, dashed,"\phi_{g,Z,Y}"]  & A_Y
\end{tikzcd}
\]

In a sense, a shape is an amalgam of small axial algebras organised around an axet in a consistent way.  Note that every $A_Y$ 
must be generated by at most two generators, so indeed each of the $A_Y$ is `small'.  The subaxet $Y$ is $1$-generated if and only if it consists of just a single axis.  In this case, $A_Y \cong \FF$ and it is usually denoted by $1\A$.

Note also that if we build a shape from an algebra, then the 
consistency requirement is automatically satisfied as $\phi_{g,Z,Y}$ is just the restriction of the action of $g$ from $A$ to 
$A_Z$.

Furthermore, $\phi_{g,Z,Y} \colon A_Z \to A_Y$ is unique and depends only on the action of $g$, as follows from the next 
lemma.

\begin{lemma}
If $g,g' \in G$ are such that $\psi_g|_{Z} = \psi_{g'}|_{Z}$, then $\phi_{g,Z,Y} = \phi_{g', Z, Y}$.
\end{lemma}
\begin{proof}
Since $ \phi_{g,Z,Y}$ and $\phi_{g',Z,Y}$ are both algebra homomorphisms from $A_Z$ to $A_Y$, it suffices to show that they agree on a generating set of $A_Z$.  Indeed, $\theta_Z(Z)$ generates $A_Z$ and for $z \in Z$, we have $\phi_{g,Z,Y}(\theta_Z(z)) = \theta_Y\circ(\psi_g(z)) = \theta_Y\circ(\psi_{g'}(z)) = \phi_{g',Z,Y}(\theta_Z(z))$.
\end{proof}

This leads to the following results.

\begin{corollary}\label{conjaxetshape}
Let $\Theta$ be a shape on an axet $X$.
\begin{enumerate}
\item For any $Z \subseteq \mathcal{X}_2$, $\phi_{1, Z, Z}$ is the identity map on $A_Z$.
\item For $Z, Y, T \in \mathcal{X}_2$ and $g, h \in G$ such that $Z^g \subseteq Y$ and $Y^h \subseteq T$, we have
\[
\phi_{gh, Z, T} = \phi_{h, Y, T} \circ \phi_{g, Z, Y}.
\]
\item If $Y, Z \in \mathcal{X}_2$ such that $Z^g = Y$ for $g \in G$, then $(\phi_{g, Z, Y})^{-1} =  \phi_{g^{-1}, Y, Z}$ and in particular, both maps are algebra isomorphisms.
\end{enumerate}
\end{corollary}
\begin{proof}
The first is immediate.  The second follows as $Z^{gh} \subseteq T$ and the composition $\phi = \phi_{h, Y, T} \circ \phi_{g, Z, Y}$ satisfies $\theta_T \circ \psi_{gh} = \phi \circ \theta_Z$ and so, by uniqueness, it equals $\phi_{gh,Z, T}$.  Finally, the last claim follows from the first two parts if we take $Y = Z^g$, $T = Z$ and $h = g^{-1}$.
\end{proof}

\begin{definition}
Suppose that $\Theta = \{ \theta_Y \colon Y \to A_Y : Y \in \mathcal{X}_2(X) \}$ and $\Theta' = \{ \theta'_{Y'} \colon Y' \to A'_{Y'} : Y' \in \mathcal{X}_2(X') \}$ are shapes on axets $X$ and $X'$.  A \emph{morphism} from $\Theta$ to $\Theta'$ is a morphism $\psi \colon X \to X'$ such that for all $Y \in \mathcal{X}_2(X)$, there exists an algebra homomorphism $\phi_{Y} \colon A_Y \to A'_{\psi(Y)}$ such that
\[
\phi_Y \circ \theta_Y = \theta'_{\psi(Y)} \circ \psi
\]
\end{definition}

In other words, the following diagram commutes.
\[
\begin{tikzcd}[row sep = large, column sep = large]
Y \arrow[r, "\psi"] \arrow[d, swap, "\theta_Y"] & \psi(Y) \arrow[d, "\theta'_{\psi(Y)}"] \\
A_Y \arrow[r, swap, dashed,"\phi_{Y}"]  & A'_{\psi(Y)}
\end{tikzcd}
\]

Note that the algebra homomorphisms $\phi_Y$ are part of the morphism of shapes.  However, they are uniquely defined by $\psi$ and so all we need to assume is that they exist.

This gives us a category of shapes on $S$-axets and when we say that two shapes are isomorphic, we mean that they are isomorphic in this category.  Such an isomorphism of shapes naturally is an isomorphism of the underlying axets.

Now that we have a good concept of isomorphism of shapes, we can develop a more concise way of describing a shape.  By part 3 of Corollary \ref{conjaxetshape}, we have a natural isomorphism between the algebras $A_Y$ for conjugate subaxets $Y$.  Moreover, if the consistency condition (\ref{consistency req}) holds for some pair of subaxets $Z \subset Y$, then it must also hold for $Z^g \subset Y^g$, for $g \in G$.  In light of this, we define the \emph{condensed shape} to be a set of embeddings $\theta_Y$, one for each representative $Y \in \mathcal{X}_2$ of the $G$-orbit $Y^G$ subject to the consistency condition.  It is clear that the (full) shape can be recovered from the condensed shape uniquely up to isomorphism.  From now on, when we talk of the shape, we usually mean the condensed shape.

Let us now note that these choices for different $Y$ are not always independent because of the consistency requirement.  In particular, if $Z \subset Y$ are two $2$-generated subaxets, then $\theta_Y$ defines $\theta_Z$ up to isomorphism.  A choice of $\theta_Z$ also restricts the possible choices for $\theta_Y$, however this is not as strong as the other way around.

Note that if $Z$ is a subaxet with a single axis, then $A_Z \cong \FF$ and the consistency requirement for $Z \subseteq Y$ is always satisfied.  In other words, $Z$ does not impose any extra conditions on the shape and we may just limit ourselves to the subaxets which are $2$-generated and not $1$-generated.  We let $\mathcal{X}_2^\sharp = \mathcal{X}_2^\sharp(X)$ be the set of all such subaxets of $X$.

We use containment to define an object which will help us better understand the consistency requirements.  Consider the digraph whose vertex set is $\mathcal{X}_2^\sharp$ and the directed edges are given by containment.  As for the shape, if $Y \supseteq Z$, then $Y^g \supseteq Z^g$ for all $g \in G$.  Hence we can quotient by the action of $G$ making a more compact object as follows.

\begin{definition}
Let $X = (G, X, \tau)$ be an $S$-axet.  The \emph{shape graph} $\Gamma_X$ of $X$ has vertices given by the $G$-orbits on $\mathcal{X}_2^\sharp$ and there is a directed edge from one orbit to another, if there exists representatives $Y$ and $Z$, respectively, such that $Y \supseteq Z$.
\end{definition}

Note that in \cite{axialconstruction}, we gave a similar definition of shape graph, however there the vertices were pairs of axes\footnote{Also, that definition was in terms of domination which meant containment between subaxets generated by the two pairs.}.  The definition we give here is better as there could be multiple pairs (up to the action of $G$) which generate the same subaxet.

Recall that the \emph{weakly connected components} of a directed graph $\Gamma$ are the connected components of the undirected graph with the same vertex set as $\Gamma$, but an undirected edge for every directed edge.  

Now, to select a valid shape for an axet $X$, we may consider its shape graph $\Gamma_X$ and give an embedding $\theta_Y$ for each vertex $Y$ of $\Gamma_X$.  The directed edges in the graph show us when we must consider the consistency condition.  In particular, for a weakly connected component of the shape graph, the choices of embeddings for the vertices will depend on one another.  Indeed, there may only be one such possible choice.  However, choices of embeddings in different weakly connected components are independent.

Finally, we wish to know if there is an axial algebra which contains the given axet and shape.

\begin{definition}\label{def:completion}
Let $X$ be an axet and $\Theta$ a shape on $X$.  A \emph{completion} of $\Theta$ is a surjective morphism $\psi \colon \Theta \to \Theta(A)$ of shapes, for some axial algebra $A$ (we may also say that $A$ is the completion of $\Theta$).  We say a completion is \emph{faithful} if $\psi$ is injective.  If there does not exist such a completion, we say that $\Theta$ \emph{collapses}.
\end{definition}

We are most interested in completions which are faithful.

It is obvious that a shape $\Theta$ for $X$ defines by restriction a shape $\Theta_Y$ on any subaxet $Y$ of $X$ and similarly, a completion for $\Theta$ induces a completion for $\Theta_Y$.  Because of this, if $\Theta$ contains a subshape $\Theta_Y$  which is not faithful, then $\Theta$ also cannot be faithful.  If we have the stronger property that $\Theta_Y$ collapses, then $\Theta$ also collapses.  So, if we can find collapsing (or non-faithful) shapes on small axets $Y$, then we can use these to show that a shape $\Theta$ on a bigger axet $X$ containing $Y$ also collapses (or is non-faithful).  In this spirit, in Section \ref{sec:forbidden}, we will give several results about collapsing small shapes.


\section{$2$-generated algebras of Jordan type}\label{sec:2gen}

In the next section, we will tackle completions of shapes on a one-point extension of an axet $X$.  For this, we will need some details about the $2$-generated algebras of Jordan type which we discuss here.

Throughout, $1\A$ will denote the $1$-dimensional axial algebra spanned by a single axis.  Also $\mathbb{F}$ will be a field of characteristic different from $2$.

\begin{definition}
A $2\B$ axial algebra is spanned by two axes $a_0$, $a_1$ such that $a_0 a_1 = 0$.
\end{definition}

It is isomorphic to $\mathbb{F} \times \mathbb{F}$ and is a $2$-generated axial algebra for many different fusion laws, including $\cJ(\eta)$ and $\cM(\alpha, \beta)$.


\subsection{Axial algebras of Jordan type $\eta$}\label{sec:Jordan}

Here we briefly review the $2$-generated algebras of Jordan type $\eta$ (for a more full exposition see \cite{Axial2}).  We use $\cJ^2(\eta)$ to denote the set of all $2$-generated axial algebras of Jordan type $\eta$.

The next example belongs to the class of Matsuo algebras corresponding to $3$-transposition groups.  It is the smallest non-trivial example.

\begin{definition}
Let $\eta \in \mathbb{F} - \{1,0\}$.  Let $3\C(\eta)$ be the algebra with basis $x,y,z$ and, for $a,b \in \{x,y,z\}$, the algebra product is defined by
\[
ab = \begin{cases}
a, & \mbox{if } b=a; \\
\frac{\eta}{2}(a + b - c), & \mbox{if } \{a, b, c\} = \{x, y, z\}.
\end{cases}
\]
This is an axial algebra of Jordan type $\eta$.
\end{definition}

Note that $x(x+y+z) = (1+\eta)x$ and similarly for $y$ and $z$.  So, if $\eta \neq -1$, then $\frac{1}{1+\eta}(x+y+z)$ is the identity for $3\C(\eta)$.  However, if $\eta = -1$, then $I:=\la x + y + z \ra$ is a nil ideal.  The quotient $3\C(-1)/I$ is an axial algebra of Jordan type $-1$ denoted by $3\C(-1)^\times$.  Note that $xy = -\frac{1}{2}(x+y-z) =z= -x-y$ in $3\C(-1)^\times$ and so $3\C(-1)^\times \not \cong 2\B$.

The only other case where $3\C(\eta)$ is not simple is for $\eta = 2$.  Namely, $3\C(2)$ has an ideal $J := \la x-y, y-z \ra$ which is the $2$-dimensional space of elements whose coefficients sum to $0$.  Note that in this case, for all axes, $J = A_0 \oplus A_2$.  Clearly, the quotient $3\C(2)/J$ is $1$-dimensional and hence isomorphic to $1\A$.  If $\ch(\FF) \neq 3$, then $I$ and $J$ are the only possible non-trivial proper ideals of $3\C(\eta)$ and so $3\C(-1)^\times$ is simple.  In characteristic $3$, $-1=2$ and so we have both ideals.  In fact, $I \subset J$ and hence $3\C(-1)^\times$ is not simple.

If $\eta \neq \frac{1}{2}$, then $2\B$ and $3\C(\eta)$ for all $\eta$ and $3\C(-1)^\times$ for $\eta = -1$ are the only axial algebras of Jordan type generated by two distinct axes.  For $\eta = \frac{1}{2}$, the situation is more complicated and there are infinitely many additional algebras.

\begin{definition}\label{def:spinfactor}
Let $V$ be a $2$-dimensional vector space over $\mathbb{F}$ and $b$ be a symmetric bilinear form on $V$.  We define a product on $S := \la \1 \ra \oplus V$ by
\[
(c\1 + u)(d\1+v) = (cd + \tfrac{1}{2}b(u,v))\1 + cu+dv
\]
where $c,d \in \mathbb{F}$ and $u,v \in V$.  The algebra $S = S(b)$ is a Jordan algebra called the \emph{spin factor} (when $\dim(V) = 2$, this is $\mathrm{Cl}^J(\mathbb{F}^2, b)$ in \cite{Axial2}).
\end{definition}

Note that two spin factor algebras $S(b)$ and $S(b')$ are isomorphic if and only if their forms $b$ and $b'$ are equivalent.  It is easy to see that the non-trivial (non-zero and non-identity) idempotents in $S$ are $x := \tfrac{1}{2}(\1 + u)$, where $u \in V$ with $b(u,u) = 2.$\footnote{We note that $V$ might not contain any vectors $u$ such that $b(u,u)=2$, in which case $S$ is not an axial algebra.}  The adjoint of $x$ has eigenspaces
\begin{align*}
A_1(x) &= \la x \ra, \\
A_0(x) &= \la  \tfrac{1}{2}(\1 - u) \ra, \\
A_\frac{1}{2}(x) &= u^\perp \subset V.
\end{align*}
Moreover $x$ is a primitive axis of Jordan type $\frac{1}{2}$ and so the Miyamoto involution $\tau_x$ acts trivially on $\la \1, u\ra$ and inverts $u^\perp$.  Note that $x^- :=  \tfrac{1}{2}(\1 - u)$ is also an axis.  We call $x^-$ the \emph{opposite axis} of $x$ and it is clear from the eigenspaces of $x$ that it is the only axis of $S$ which satisfies $x x^- = 0$.  In this case, $\lla x, x^- \rra \cong 2\B$ and it is a proper subalgebra of $S$.

From now on, suppose that $x = \tfrac{1}{2}(\1 + u)$, $y = \tfrac{1}{2}(\1 + v)$ are two distinct non-opposite axes in $S$.  Then $u, v$ is a basis for $V$ satisfying $b(u,u) = 2 = b(v,v)$.  The form $b$ is fully identified by the value $\delta := b(u,v)$ and so we will write $S(\delta)$ for $S(b)$.

We wish to investigate when $x$ and $y$ generate $S$.

\begin{proposition}\label{Jordanproper}
We have $\lla x, y \rra \neq S$ if and only if $\delta =2$.  In this case, $xy = \frac{1}{2}(x+y)$ and $\la x, y \ra \lhd S$.
\end{proposition}
\begin{proof}

Since $x$ and $y$ are not opposite, $u$ and $v$ are linearly independent.  In particular, $\1 \not \in \la x, y \ra$.  Since the product $xy = \tfrac{1}{4} \big( (1 + \tfrac{1}{2}b(u,v))\1 + u+v \big)
= \tfrac{1}{2}(x+y) + \tfrac{1}{8}(b(u,v) - 2)\1$, we see that $xy \in \la x,y \ra$ if and only if $ \delta = b(u,v)=2$ and so $xy = \frac{1}{2}(x+y)$.  Since $S= \la \1\ra \oplus \la x,y\ra$, $S \la x, y \ra = \la x, y \ra$ and hence $\la x, y \ra \lhd S$.
\end{proof}

We denote the proper subalgebra $\lla x,y\rra$ arising in Proposition \ref{Jordanproper} when $\dl = 2$ by $S(2)^\circ$ (this is $\mathrm{Cl}^0(\mathbb{F}^2, b)$ in \cite{Axial2}).

Notice that as long as we have two distinct non-opposite axes, $x$ and $y$, then $S$ is generated by $x$ and $y$, by $x$ and $y^-$, or by both.  We now restrict ourselves to where $S(\dl) = \lla x, y \rra$ is a $2$-generated axial algebra of Jordan type $\frac{1}{2}$, which by the above is precisely when $\dl \neq 2$.  

Every axial algebra of Jordan type admits a Frobenius form \cite{HSS} and the form for $S$ is a natural extension of $b$.  We will also denote it by $b$ and then $b(\1, \1) = 2$ and $b(\1, v) = 0$ for all $v \in V$.  Note that $b(z,z) = 1$ for each axis $z$.

Note that the radical is non-trivial if and only if $\dl = -2$ (recall that we excluded $\delta = 2$).  In this case, the radical of the form $b$ is a $1$-dimensional ideal $R :=  \la u+v \ra = \la x -y^-\ra = \la x^- - y \ra$ and $R^2=0$.  Note that $b(u,-v) = 2$ and so in $S(-2)$, $\lla x, y^- \rra \cong S(2)^\circ \cong \lla x^-,y \rra$.  We note that these are ideals as $S = \la \1 \ra \oplus \lla x,y^- \rra = \la \1 \ra \oplus \lla x^-,y \rra$.  Furthermore, their intersection is $R$.

We will see now that this is the only case where $S$ is not simple.

\begin{proposition}\label{spinsimple}\textup{\cite[Theorem 4.7]{Axial2}}
Suppose $S= \lla x,y \rra$ is a spin factor algebra as above.  Then $\delta \neq 2$ and
\begin{enumerate}
\item $S$ is simple if and only if $\delta \neq - 2$.
\item If $\delta = - 2$, then $S$ has a single $1$-dimensional ideal $R = \la x-y^-\ra$, equal to the radical of $b$, and two $2$-dimensional ideals corresponding to the two $1$-dimensional ideals in $S/R \cong 2\B$.  These two ideals are isomorphic to $S(2)^\circ$ and they are $\lla x, y^- \rra$ and $\lla x^-, y \rra$.
\end{enumerate}
\end{proposition}

Clearly, if $J \unlhd S(-2)$ is $2$-dimensional, then $S(-2)/J \cong 1\A$.  Also, note that every axis $a$ lies in one of the two $2$-dimensional ideals in $S(-2)$, say $a \in J$.  Then, $a^-$ cannot be in $J$ as $\lla a, a^- \rra \cong 2\B$, so $a^-$ lies in the other $2$-dimensional ideal, which we denote $J^-$.

We now want to examine the situation in Lemma \ref{Jordanproper} in a bit more detail for use later.

\begin{lemma}\label{spintau}
Suppose $S= \lla x,y \rra$ is a spin factor algebra as above and set $y' = y^{\tau_x}$.  Then $\lla y, y' \rra$ is $2$-dimensional if and only if
\begin{enumerate}
\item $\dl= 0$, $y' = y^-$ and $\lla y, y' \rra \cong 2\B$, or
\item $\dl = -2$ and $\lla y, y' \rra \cong S(2)^\circ$.
\end{enumerate}
\end{lemma}
\begin{proof}
The algebra $\lla y, y' \rra$ is $2$-dimensional if and only if either $y'$ is opposite to $y$, in which case $\lla y, y' \rra \cong 2\B$, or, by Lemma \ref{Jordanproper}, $\lla y,y' \rra \cong S(2)^\circ$.

Suppose first that $y' = y^-= \frac{1}{2}(\1 -v)$.  Since $A_\frac{1}{2}(x) = u^\perp$, for $\tau_x$ to switch $y$ and $y'$, we must have $v-(-v) = 2v \in u^\perp$.  That is, $b(u,v)=0$.  
 
Finally, suppose that $J := \lla y, y' \rra \cong S(2)^\circ$.  As noted, $J \unlhd S$ and so, by Proposition \ref{spinsimple}, $b(u,v) = -2$.  
\end{proof}

There is one further example we need to consider.

\begin{definition}
Let $\Cl = \la x,y,z \ra$ be the $3$-dimensional algebra with multiplication given by $x^2 = x$, $y^2=y$, $az = 0$ for all $a \in \Cl$, and
\[
xy = \tfrac{1}{2}(x+y) + z
\]
\end{definition}

Note that we write $\Cl$ since it is a cover of $S(2)^\circ$ by extending by a nil element $z$.  ($\Cl$ was called $\mathrm{Cl}^{00}(\mathbb{F}^2, b)$ in \cite{Axial2}).

It is an easy calculation to see that $a$ is an idempotent in $\Cl$ if and only if $a = \mu x + (1-\mu) y + 2\mu(1-\mu)z$ for some $\mu \in \mathbb{F}$.  Moreover, 
\begin{align*}
A_0(a) &= \la z \ra \\
A_\frac{1}{2}(a) &= \la x-y -2(1-2\mu)z \ra
\end{align*}
and so $a$ is a primitive axis of Jordan type $\frac{1}{2}$.  It is then an easy calculation to see that any pair of distinct axes generate $\Cl$, so in particular it is an axial algebra of Jordan type $\frac{1}{2}$.

The algebra $\Cl$ has a Frobenius form given by $(a,b) = 1$, for any two axes $a$ and $b$, and $(a,z)=0=(z,z)$.  By Lemma \ref{projundirected}, any proper ideal of $\Cl$ is contained in the radical and, by Theorem \ref{radical}, an easy calculation shows that this is $\la x-y, z\ra$.  Observe that $\la z \ra$ is a nil ideal and from the multiplication, one can show that $\Cl/\la z\ra \cong S(2)^\circ$.  A straightforward calculation shows that these are the only non-trivial proper ideals.


We now state the classification of $2$-generated axial algebras of Jordan type $\eta$.

\begin{theorem}\textup{{\cite[Theorem 1.1]{Axial1}}}\label{Jordanclassification}
Let $A \not \cong 1\A$ be a $2$-generated axial algebra of Jordan type $\eta$.  Then $A$ is isomorphic to one of
\begin{enumerate}
\item $2\B$,
\item $3\C(\eta)$,
\item $3\C(-1)^\times$ and $\eta = -1$,
\item $S(\delta)$, $\delta \neq 2$ and $\eta = \frac{1}{2}$,
\item $S(2)^\circ$ and $\eta = \frac{1}{2}$,
\item $\Cl$ and $\eta = \frac{1}{2}$.
\end{enumerate}
\end{theorem}

\begin{remark}\label{Jordaniso}
The only possible isomorphism between two algebras in different cases above is between $3\C(\frac{1}{2})$ and $S(\delta)$ for some specific values of $\delta$ and choice of axes.  Indeed, since both algebras have an identity, the identity and two generating axes are a basis.  So writing $\lla a, b \rra \cong 3\C(\frac{1}{2})$ and $\lla x, y \rra \cong S(\delta)$, the map $\phi\colon \1 \mapsto \1, a \mapsto x, b \mapsto y$ is an isomorphism if and only if $\phi(ab) = xy$.  In both algebras $ab = \frac{1}{2}(a+b) + \gamma \1$ for some $\gamma \in \mathbb{F}$.  So, to check for an isomorphism we just need to identify $\gamma$ for each algebra and pair of generating axes.

In $\lla x, y\rra = S(\delta)$, for $x, y$ we get $\gamma = \frac{1}{8}(\delta -2)$.  The algebra $3\C(\frac{1}{2})$ has a $1$-dimensional variety of idempotents, but there are two obvious classes of $\cJ(\frac{1}{2})$-axes, namely $\{a,b,c\}$ and $\{\1-a, \1-b, \1-c\}$.   One can easily see that $a(\1-a) = 0$, and so $\lla a, \1-a \rra \cong 2\B$ and similarly for $b$ and $c$.  Apart from these, any other pair generate $3\C(\frac{1}{2})$.  A short calculation then shows that using $a,b$, or $\1-a, \1-b$, we get $\gamma = -\frac{3}{8}$ and so $3\C(\frac{1}{2}) \cong S(-1)$; and using $a, \1-b$, we get $\gamma = -\frac{1}{8}$ and so $3\C(\frac{1}{2}) \cong S(1)$.
\end{remark}

\begin{remark}\label{onlyspin}
If $A \in \cJ^2(\eta)$ and $x, y \in A$ are distinct axes such that $\lla x, y \rra \neq A$, then $A = S(\dl)$ and the situation is described in Proposition \ref{Jordanproper}.
\end{remark}


\subsection{Axets}\label{sec:jordanaxet}

It is clear that $2\B$ has axet $X(2)$.  The algebra $3\C(\eta)$, $\eta \neq \frac{1}{2}$, has axet $X(3)$.

For $\eta=\frac{1}{2}$, where we have $\Cl$ and the spin factor algebras $S(\dl)$ (which include $3\C(\frac{1}{2})$), the situation is more complicated.  First let $A = \lla x,y\rra \cong S(\dl)$.  Recall that $x = \frac{1}{2}(\1+u)$ and $y = \frac{1}{2}(\1+v)$, where $b(u,u)=2=b(v,v)$ and $b(u,v) =\dl$.  For an axis $z = \frac{1}{2}(\1+w)$ in $S(\dl)$, the Miyamoto involution $\tau_z$ restricted to the underlying $2$-dimensional quadratic space is $-r_w$, where $r_w$ is the reflection in $w$.  So to calculate the closure of $\{ x, y\}$ under the action of the Miyamoto group, it suffices to consider the closure of $\{u, v\}$ under the action of  $\la -r_u, -r_v \ra \cong \Miy(A)$.

As we will see later, this more general setup of a $2$-dimensional reflection group arises also in the split spin factor algebras \cite{splitspin} and we briefly recount it here.  Note that in \cite{splitspin}, the form was scaled differently, with $b' = \frac{1}{2}b$, but here we will use our $b$.  Let $D$ be the dihedral subgroup $D := \la -r_u, -r_v \ra$.  We are interested in the size of $\Omega := u^D \cup v^D$.  Note that $u^D$ and $v^D$ may be equal, or disjoint.  Define $\theta$ to be the involution swapping $u$ and $v$.  Then the dihedral group $\hat D := \la \theta, -r_u \ra$ contains $D$ and $|\hat D : D | \leq 2$.  We note that $u^{\hat D} = u^D \cup v^D = \Omega$.  Since the stabiliser in $\hat D$ of $u$ is $\la -r_u \ra$, the size of $\Omega$ is equal to the order of $\rho := \theta(-r_u)$.  With respect to the basis $u, v$ we have 
\[
\rho = \begin{pmatrix} \dl & -1 \\ 1 & 0 \end{pmatrix}
\]
Since $\rho$ has determinant $1$ and trace $\dl$, it has eigenvalues $\zeta$ and $\zeta^{-1}$ which are roots of the polynomial $x^2 -\dl x +1$.  These eigenvalues lie in $\FF$, or a quadratic extension of $\FF$.  In the next lemma, $o(\zeta)$ denotes the multiplicative order of $\zeta$.  In particular, $o(\zeta)$ is finite if and only if $\zeta$ is a root of unity.

\begin{lemma}\textup{\cite[Lemma 5.3]{splitspin}}\label{quadaxet}
\begin{enumerate}
\item If $\zeta$ is not a root of unity, then $|\Omega|$ is infinite.
\item If $\zeta \neq \pm 1$, then $|\Omega| = o(\zeta)$.
\item Suppose that $\zeta = \pm 1$.  If $\ch(\FF) = 0$, then $|\Omega| = \infty$; if $\ch(\FF) = p >0$, then $|\Omega| = p$ if $\zeta = 1$ and $2p$ if $\zeta = -1$.
\end{enumerate}
\end{lemma}

Using the above lemma and since $\dl = \zeta + \zeta^{-1}$, the set of $\dl$ for which $\rho$ has finite order $n$ is
\[
N(n) := \begin{cases}
\{  \zeta + \zeta^{-1} \in \FF : \zeta \in \bar{\FF}, o(\zeta) = n \}, & \mbox{if } n \neq p, 2p \\
\{ 2\}, & \mbox{if } n = p >0 \\
\{ -2\}, & \mbox{if } n = 2p >0 \\
\end{cases}
\]
where $\ch(\FF) = p$.\footnote{Note that in \cite{splitspin} $N(n)$ also appeared, but the values there in each set $N(n)$ are scaled by a factor of $\frac{1}{2}$.}  When $\ch(\FF)=0$, $n >1$ can be arbitrary, whereas if $\ch(\FF) = p >0$, then $n>1$ can be $p$, $2p$, or coprime to $p$.  Note that $N(n)$ and $N(m)$ are disjoint for $n \neq m$.  We define $N(\infty)$ to be $\FF - \bigcup_{n>1} N(n)$.

We can now describe the axet for $S(\dl) = \lla x, y\rra$, where as before $x = \frac{1}{2}(\1+u)$, $y = \frac{1}{2}(\1+v)$.  We have the following immediately from Lemma \ref{quadaxet} and the definition of $N(n)$.

\begin{lemma}\label{spinfactoraxet}
$S(\dl)$, $\dl \neq 2$, has axet $X(n)$, $n \in \N \cup \{ \infty \}$, if and only if $\dl \in N(n)$.  $S(2)^\circ$ has axet $X(\infty)$ in characteristic $0$ and $X(p)$ in characteristic $p>0$.
\end{lemma}

Finally, for $\Cl$ we have the following.

\begin{lemma}\label{Cl00axet}
$\Cl$ has axet $X(\infty)$ if $\ch(\FF)=0$ and axet $X(p)$ if $\ch(\FF)=p >0$.
\end{lemma}
\begin{proof}
Let $X = x^{\Miy(X)} \cup y^{\Miy(X)}$ be the closed set of axes.  It can be seen that the difference of two distinct axes is never in the nil ideal $\la z \ra$ and so axes in $\Cl$ are in bijection with those in the quotient $\Cl/\la z \ra \cong S(2)^\circ$.  Since $\Miy(X)$ acts faithfully on $X$, it has an induced faithful action on $S(2)^\circ$ and the result follows from Lemma \ref{spinfactoraxet}.
\end{proof}


\section{Identifying shapes on $3$-generated axets}\label{sec:reduction}

The first interesting case for identifying completions of shapes on axets is that of $3$-generated axets.  In this paper, we consider a $3$-generated $C_2$-axet $X = (G, X, \tau)$ which is a one point extension of a $2$-generated axet, where the Miyamoto group has two orbits on $X$.  By Proposition \ref{1+n}, either $X$ is isomorphic to $X_1(1+n)$, with $n$ odd, or to $X_2(1+n)$, with $n$ twice odd.  Since we are primarily interested in the Monster type fusion law $\cM(\al,\bt)$ and all known $2$-generated algebras of this type are regular (see Definition \ref{def:skew}), we restrict ourselves to those axets which are regular.  In particular, $X_2(1+n)$ is skew, so here we concentrate on $X_1(1+n)$, $n$ odd, which is regular.

Note that, if $n=1$, then we do not have three axes and so $X$ is in fact $2$-generated, so we will assume that $n \geq 3$.

We now consider the shape graph on $X = X_1(1+n)$.  By definition, there exists a unique axis $a \in X$ which is fixed by the Miyamoto group and $a$ necessarily lies in any generating set of $X$ (see Section \ref{sec:1+n}).  Suppose that $X = \la a,b,c \ra$.  Then $Z := \la b,c \ra = X - \{ a\}$ is a $2$-generated subaxet.  Since $Z$ is $2$-generated, the choice of the algebra $A_Z$ on $Z$ in the shape determines the algebra $A_W$ for all subaxets $W$ of $Z$.  So the only part of the shape which is not identified by $A_Z$ is $A_Y$, where $Y = \la a, b \ra$.  Note that every subaxet $\la a, z\ra$, for $z \in Z$, is conjugate to $\la a,b \ra$.  Hence we do indeed just have two choices, $A_Y$ and $A_Z$.

Since $a$ is a fixed axis, $\tau_a =1$.  So in any completion $A$, the axis $a$ has trivial $\beta$-eigenspace\footnote{Recall that we identify an axis $a$ from the axet with its image in the algebra.}.  Hence $a$ is actually an axis of Jordan type $\alpha$.  The related Miyamoto automorphism that negates the $\alpha$-eigenspace will be denoted by $\sigma_a$ to distinguish it from $\tau_a$.  Note that we may still have $\sigma_a=1$, equivalently $A_\alpha(a) = 0$.

In the following proposition we consider completions for several different shapes $\Theta$ on $X \cong X(1+n)$ simultaneously, for all $n \geq 3$ odd.  For $z \in Z$, we define $z' = z^{\sg_a}$ and set $Z' = \{ z' : z \in Z \}$.

Recall from the introduction, that a $2$-generated algebra $\lla a,b \rra$ is called \emph{symmetric} whenever it admits an automorphism of order $2$ switching the two generating axes $a$ and $b$.

\begin{proposition}\label{1+nshapeprop}
Let $X = \la a,b,c \ra \cong X_1(1+n)$, $n \geq 3$ odd, be a $3$-generated axet, where $a$ is the fixed axis.  Let $Y = \la a,b\ra$ and $Z = X - \{a\}$.  Suppose that $\Theta$ is a shape on $X$ for $\cM(\alpha, \beta)$ and $A$ is a completion.
\begin{enumerate}
\item If $\sigma_a = 1$, then $A_Y \cong 2\B$ and $A \cong 1\A \oplus \lla b,c \rra$, where $\lla b,c \rra$ is symmetric.
\item If $\sigma_a \neq 1$, then $A_Y \in \cJ^2(\alpha) -\{ 2\B\}$ and $A$ contains an axet $\{a\} \cup Z \cup Z' \cong X_2(1+2n)$ with Miyamoto map $t$ given by $t_a = \sg_a$ and $t_z = t_{z'} = \tau_z$ for all $z \in Z$.  Moreover, $B := \lla b', c \rra$ is a symmetric $2$-generated subalgebra with $X(B) = Z \cup Z'$ and
\begin{enumerate}
\item $A = B$; or
\item $\alpha = \frac{1}{2}$, $A_Y \cong S(-2)$ and $\lla b, b' \rra \cong S(2)^\circ$; or
\item $\alpha = \frac{1}{2}$, $A_Y \cong S(0)$, $b' = b^-$ and $\lla b, b' \rra \cong 2\B$.
\end{enumerate}
\end{enumerate}
\end{proposition}
\begin{proof}
Suppose first that $\sg_a=1$.  Then $A_\al(a) = \{ 0 \}$ and so $A = \la a \ra \oplus A_0(a)$ where both summands are subalgebras which annihilate each other.  Since $b$ and $c$ are primitive and not equal to $a$, by \cite[Lemma 5.9(2)]{axialstructure}, $b, c \in A_0(a)$.  Therefore, $A_0(a) = \lla b, c \rra$, $A \cong 1\A \oplus \lla b,c \rra$ and $A_Y \cong 2\B$.

Now assume that $\sg_a \neq 1$ and so $A_Y \in \cJ^2(\alpha) -\{ 2\B\}$.  Note that $\tau_b$ fixes $a$ and so $\sg_a^{\tau_b}=\sg_{a^{\tau_b}}=\sg_a$.  That is, $\sg_a$ and $\tau_b$ commute.  Similarly, $\sg_a$ and $\tau_c$ commute and so $\sg_a$ centralises $\Miy(A) = \la \tau_b, \tau_c \ra$.  Since $Z$ is a $\Miy(A)$-orbit of odd length, $\sg_a$ cannot preserve $Z$ and so $Z$ and $Z'$ are disjoint.  Since $\tau_{z'}=\tau_{z^{\sg_a}}=\tau_z^{\sg_a}=\tau_z$, we see that $\{a\} \cup Z \cup Z' \cong X_2(1+2n)$ with Miyamoto map $t$ as required.

Let $B := \lla b', c \rra$.  Then $B$ is invariant under $\la \tau_{b'}, \tau_c \ra = \la \tau_b, \tau_c \ra\cong \Miy(A)$.  Since $\Miy(A)$ is transitive on $Z$ and so also on $Z'$, the closure of $\{b', c \}$ is $X(B) = Z \cup Z'$.  Moreover, as $\sg_a$ switches $z$ and $z'$, it leaves $B$ invariant and so it is an automorphism of $B$.  Since $n$ is odd, there exists an involution $g \in \Miy(B) \cong D_{2n}$ such that $b^g = c$.  Then $\sg_ag$ is an involution switching $b'$ and $c$ and so $B$ is a symmetric $2$-generated subalgebra.

If $\lla b, b' \rra  = A_Y$ then $a \in B$ and hence $A = B$ is $2$-generated, which is case (a).  So suppose that $\lla b, b' \rra$ is a proper subalgebra of $A_Y$.  In particular, since $\lla b, b' \rra$ is at least $2$-dimensional, by Remark \ref{onlyspin}, $\alpha = \frac{1}{2}$ and $A_Y$ is a spin factor algebra.  Moreover, by Lemma \ref{spintau}, either $A_Y \cong S(0)$, $b' = b^-$ and $\lla b, b' \rra \cong 2\B$, which is case (c); or $A_Y \cong S(-2)$ and $\lla b, b' \rra \cong S(2)^\circ$, which is case (b).
\end{proof}

We will use the above result as a powerful reduction theorem.  Apart from the two exceptional cases, $A$ is either a direct sum of a symmetric $2$-generated algebra with a $1\A$ algebra, or a symmetric $2$-generated algebra on an $X(2n)$ axet.  Since the algebras in the outcome of this reduction theorem are symmetric, the statement can be further enhanced by using the classification of symmetric $2$-generated algebras of Monster type and checking them class by class.  In the remainder of this section, we develop a useful condition coming from Proposition \ref{1+nshapeprop} that we can use when going through the list of symmetric algebras.

In case 1, the algebra $\lla b,c \rra$ must have axet $X(n)$, with $n$ finite and odd.  In case 2(a), the algebra $B$ must have axet $X(2n)$, $n$ finite and odd, and satisfy the following property:

\begin{enumerate}
\item[(J)] There exists an additional axis $a$ of Jordan type $\al$, which is fixed by $\Miy(X)$, such that $\sg_a$ switches the two halves of $X \cong X(2n)$. \label{J}
\end{enumerate}

Note that, in such an algebra, given $x \in X$, we can identify $x'$ without knowing $a$, or $\sg_a$.  Indeed, $x'$ is the opposite axis in the axet $X$.  That is, the unique axis such that $\tau_x = \tau_{x'}$.  The following lemma gives a condition which we will use to identify which algebras have property (J).

\begin{lemma}\label{propertyJlem}
Suppose that $A$ is a symmetric $2$-generated axial algebra of Monster type which has property \textup{(}J\textup{)}.  Then either the additional axis $a$ of Jordan type $\al$ is contained in $\bigcap_{x \in X} \lla x, x' \rra$ \textup{(}recall that $x' = x^{\sg_{a}}$\textup{)}; or
\begin{enumerate}
\item for every $x \in X(A)$, $\lla x, x'\rra \cong 2\B$ and $\lla a, x \rra \cong S(0)$; or
\item for every $x \in X(A)$, $\lla x, x'\rra \cong S(2)^\circ$ and $\lla a, x \rra \cong S(-2)$.
\end{enumerate}
In particular, $\al = \frac{1}{2}$ in these two exceptional cases.
\end{lemma}
\begin{proof}
Since $a$ is fixed by $\Miy(X)$, $\lla a,x \rra$, for $x \in X$, is a subalgebra with trivial Miyamoto group.  Equivalently, it is an axial algebra of Jordan type $\al$.  Since $\tau_x$ fixes $a$, $\sg_a^{\tau_x} = \sg_{a^{\tau_x}} = \sg_a$ and so $\tau_x$ and $\sg_a$ commute.  Hence, $\tau_{x^{\sg_a}} = \tau_x$.  However, there is only one other axis in $X$ with the same Miyamoto involution as $x$, namely the opposite axis $x'$.  Since $\sg_a$ switches the two orbits, $x^{\sg_a} = x'$ and so $x' \in \lla a, x \rra$.  Clearly, $\lla x, x' \rra \leq \lla a, x \rra$, as $x' = x^{\sg_a}$.  Moreover this containment is proper if and only if $\alpha= \frac{1}{2}$ and $\big(\lla x, x' \rra, \lla a, x \rra\big)$ is isomorphic to either $\big(2\B, S(0)\big)$, or to $\big(S(2)^\circ, S(-2)\big)$.  Therefore, if $\lla x, x' \rra \not \cong 2\B, S(2)^\circ$, then $a$ must lie in $\bigcap_{x \in X} \lla x, x' \rra$.  
\end{proof}

In the next section, we identify the algebras which have property (J).  We will discuss the two exceptional situations, 2(b) and 2(c), from Proposition \ref{1+nshapeprop} in Section \ref{sec:forbidden}.


\section{Symmetric $2$-generated $\cM(\alpha, \beta)$-axial algebras}\label{sec:2genalbt}

A symmetric $2$-generated axial algebra $A = \lla a,b \rra$ is one where there is an involutory automorphism $f$, called the \emph{flip}, which switches the generating axes $a$ and $b$.  Since the algebra is symmetric, its axet $X = a^{\Miy(A)} \cup b^{\Miy(A)}$ is regular (cf. Definition \ref{def:skew}).  It will be convenient below to number the axes in a standard way.  We let $a_0 = a$, $a_1 = b$ and for $k \in \Z$, we let $a_{2k+\epsilon} = {a_\epsilon}^{\rho^k}$, where $\epsilon = 0,1$ and $\rho = \tau_a \tau_b$.  Our notation does not mean that every axet is infinite, rather if the axet is finite, then an axis will have multiple names with $a_i = a_j$ if and only if $i \equiv j \mod |X|$.

Let us now turn to the statement of the classification of the symmetric case.

\begin{theorem}\textup{\cite{yabe, highwater5, HWquo}}\label{2genMonster}
A symmetric $2$-generated $\cM(\al, \bt)$-axial algebra is one of the following:
\begin{enumerate}
\item an axial algebra of Jordan type $\al$, or $\bt$;
\item a quotient of the Highwater algebra $\cH$, or its characteristic $5$ cover $\hatH$, where $(\al,\bt) = (2, \frac{1}{2})$; or
\item one of the algebras listed in \textup{\cite[Table 2]{yabe}}.
\end{enumerate}
\end{theorem}

We organise the list in \cite[Table 2]{yabe} into twelve families, each depending on its own set of parameters. Below we split the families into two main types: those with a finite and fixed axet and those where the axet is generically infinite but can be finite depending on the parameters.

\begin{enumerate}
\item $3\A(\al,\bt)$, $4\A(\frac{1}{4}, \bt)$, $4\B(\al, \frac{\al^2}{2})$, $4\J(2\bt, \bt)$, $4\Y(\frac{1}{2}, \bt)$, $4\Y(\al, \frac{1-\al^2}{2})$, \newline$5\A(\al, \frac{5\al-1}{8})$, $6\A(\al, \frac{-\al^2}{4(2\al-1)})$, $6\J(2\bt, \bt)$ and $6\Y(\frac{1}{2}, 2)$;
\item $\IY_3(\al, \frac{1}{2}, \mu)$ and $\IY_5(\al, \frac{1}{2})$.\footnote{Yabe's names for these algebras are: (1) $\textrm{III}(\al, \bt,0)$, $\textrm{IV}_1(\frac{1}{4}, \bt)$, $\textrm{IV}_2(\al, \frac{\al^2}{2})$, $\textrm{IV}_1(2\bt, \bt)$, $\textrm{IV}_2(\frac{1}{2}, \bt)$, $\textrm{IV}_2(\al, \frac{1-\al^2}{2})$, $\textrm{V}_1(\al, \frac{5\al-1}{8})$, $\textrm{VI}_2(\al, \frac{-\al^2}{4(2\al-1)})$, $\textrm{VI}_1(2\bt, \bt)$ and $\textrm{IV}_3(\frac{1}{2}, 2)$ and (2) $\textrm{III}(\al, \frac{1}{2}, \dl)$, where $\dl = -2\mu-1$, and $\textrm{V}_2(\al, \frac{1}{2})$.  Here the Roman numeral indicates axial dimension.}
\end{enumerate}

Note that not all these algebras exist for all parameters.  For an algebra here to be an axial algebra, it is clear that $1$, $0$, $\al$ and $\bt$ must be distinct.  For all the algebras except $3\A(\al, \bt)$ and $6\A(\al, \frac{-\al^2}{4(2\al-1)})$, these are the only restrictions.  For the two exceptional families, we must also exempt $\al = \frac{1}{2}$.

In our notation, those algebras with a finite axet $X(n)$ we label $n\L$, and for those which generically have the infinite axet $X(\infty)$ we use $\mathrm{I}\L_d$ instead, for various letters $\L$.  This notation better fits with our focus on axets, rather than Yabe's focus on axial dimension $d$.  If the algebra can be specialised to a Norton-Sakuma algebra (see, for example, \cite[Table 3]{IPSS}), we use the corresponding letters for this, e.g. $\A$, $\B$, or $\C$.  For those which arise from Joshi's double axis construction \cite{j, doubleMatsuo}, we use $\J$ (these necessarily have $\al = 2\bt$) and for those which were found by Yabe, we use $\Y$.

We need to find among this list all the algebras which satisfies the conditions stated at the end of the last section.  That is, those with finite odd axet and those with axet $X(2n)$, $n \geq 3$ odd, which have property (J).  Those algebras which always have a finite odd axet can be read off the list above.  Namely, $3\A(\al,\bt)$ and $5\A(\al, \frac{5\al-1}{8})$.  Those which always have a finite axet $X(2n)$, $n$ odd necessarily have axet $X(6)$ and are $6\A(\al, \frac{-\al^2}{4(2\al-1)})$, $6\J(2\bt, \bt)$ and $6\Y(\frac{1}{2}, 2)$.  For those algebras, which generically have an infinite axet, we must determine when this can be finite and when the algebra has property (J).


\subsection{Jordan type algebras}

The axets for Jordan type algebras were discussed in Subsection \ref{sec:jordanaxet}, in particular when they are finite (cf. Lemmas \ref{spinfactoraxet} and \ref{Cl00axet}).

We now discuss property (J).  The only algebras which possibly have axet $X(2n)$, for $n \geq 3$ odd, are some spin factor algebras.  Since every non-trivial idempotent in a spin factor algebra has eigenvalues $1$, $0$ and $\frac{1}{2}$ (see discussion after Definition \ref{def:spinfactor}), there does not exist any additional axis of Jordan type $\al$ with $\al \neq \frac{1}{2}$.  Hence no $2$-generated axial algebra of Jordan type $\bt$ has property (J).


\subsection{Algebras with axet $X(6)$}\label{sec:X6}

By Theorem \ref{2genMonster}, there are three families of algebras which always have the axet $X(6)$, namely $6\A(\al, \frac{-\al^2}{4(2\al-1)})$, $6\J(2\bt, \bt)$ and $6\Y(\frac{1}{2}, 2)$, and we give these in Table \ref{tab:multX6}.  We only give some products, the remaining can be obtained from using the action of the Miyamoto group.  Since the axet $X(n)$ was defined as the vertex set of the $n$-gon, the axes are numbered naturally by the integers taken modulo $n$, and here we have $n=6$.  Note that we choose different bases to those in \cite{yabe} to try to better reflect the structure of the algebras\footnote{Compared to Yabe's bases, for $6\A(\al, \frac{-\al^2}{4(2\al-1)})$ we have $c = \hat{a}_i + \hat{a}_{i+3} - \frac{2}{\al}\hat{a}_i \hat{a}_{i+3}$, $z = -\frac{3\al-2}{\al}\sum_{i = -2}^3 \hat{a}_i + \frac{8(2\al-1)(3\al-2)(5\al-2)}{\al^3(9\al-4)}\hat{p}_1 - \frac{8(2\al-1)}{\al(5\al-2)} \hat{q}$. For $6\J(2\bt, \bt)$, we have $u = \hat{a}_i + \hat{a}_{i+3} - \frac{2}{\al}\hat{a}_i \hat{a}_{i+3}$, $w = 2(\hat{a}_i + \hat{a}_{i+1}) - \frac{4}{\al}\hat{a}_i \hat{a}_{i+1}$ as per the basis in \cite{doubleMatsuo}.  For $6\Y(\frac{1}{2},2)$, we have $a_4 = a_0+a_2 - a_0a_2$, $d = a_{-1}-a_2$, $z = q$.}.

For $6\J(2\bt, \bt)$, the only values of $\beta$ we need to exclude (forbidden values) are those which cause the eigenvalues to coincide, namely $\bt \neq 0,\frac{1}{2},1$.  For $6\A(\al, \frac{-\al^2}{4(2\al-1)})$, the eigenvalues coincide when $\al = 0,1, \frac{4}{9}, -4 \pm 2\sqrt{5}$ and hence these are forbidden.  In addition, clearly we must exclude $\al=\frac{1}{2}$ for $\bt = \frac{-\al^2}{4(2\al-1)}$ to make sense.  These are the only forbidden values for $6\A(\al, \frac{-\al^2}{4(2\al-1)})$.  In particular, in our new basis for $6\A(\al, \frac{-\al^2}{4(2\al-1)})$, $\al = \frac{2}{5}$ (excluded in \cite[Table 2]{yabe} and \cite[Table 2]{gendihedral}) and $\al = \frac{1}{3}$ (excluded in \cite[Table 2]{gendihedral}) both lead to valid algebras.

\begin{table}[h!tb]
\setlength{\tabcolsep}{4pt}
\renewcommand{\arraystretch}{1.5}
\centering
\footnotesize
\begin{tabular}{c|c|c}
Type & Basis & Products \& form \\ \hline
$6\A(\al, \frac{-\al^2}{4(2\al-1)})$ & $a_{-2}, \dots, a_3, c, z$ &
 \begin{tabular}[t]{c}
  {$\begin{aligned} a_i a_{i+1} &= \tfrac{\bt}{2}(a_i + a_{i+1} - a_{i+2} - a_{i+3} \\ 
    &\phantom{{}= \tfrac{\bt}{2}({}} {} - a_{i-1} - a_{i-2}+c+z) \end{aligned}$} \\
  $a_i a_{i+2} = \frac{\al}{4}(a_i + a_{i+2}) + \frac{\al(3\al-1)}{4(2\al-1)}a_{i+4} -\frac{\al(5\al-2)}{8(2\al-1)}z$ \\
  $a_ia_{i+3} = \frac{\al}{2}(a_i + a_{i+3} -c)$ \\
  $a_i c = \frac{\al}{2}(a_i + c - a_{i+3})$ \\
  $a_i z = \frac{\al(3\al-2)}{4(2\al-1)}(2a_i -a_{i-2} - a_{i+2} +z)$ \\
  $c^2 = c$, $cz = 0$, $z^2 = \frac{(\al+2)(3\al-2)}{4(2\al-1)}z$ \\
  $(a_i, a_i) = (c,c) = 1$, $(a_i, a_{i+1}) = -\frac{\al^2(3\al-2)}{(4(2\al-1))^2}$, \\
  $(a_i, a_{i+2}) = \frac{\al(21\al^2-18\al+4)}{(4(2\al-1))^2}$, $(a_i, a_{i+3}) = (a_i, c) = \frac{\al}{2}$\\
  $(a_i,z) = \frac{\al(7\al-4)(3\al-2)}{8(2\al-1)^2}$\\
  $(c, z) = 0$, $(z,z) = \frac{(\al+2)(7\al-4)(3\al-2)}{8(2\al-1)^2}$
  \vspace{4pt}
  \end{tabular}
\\
$6\J(2\bt, \bt)$ & $a_{-2}, \dots, a_3, u, w$ &
  \begin{tabular}[t]{c}
  $a_ia_{i+1} = \frac{\bt}{2}(2(a_i + a_{i+1}) -w)$ \\
  $a_i a_{i+2} = \frac{\bt}{2}(a_i + a_{i+2} - a_{i+4})$ \\
  $a_ia_{i+3} = \frac{\al}{2}(a_i + a_{i+3} -u)$ \\
  $a_i u = \frac{\al}{2}(a_i + u - a_{i+3})$ \\
  $a_i w = \frac{\al}{2}(2a_i -a_{i-1} - a_{i+1} + w)$ \\
  $u^2 = u$, $uw = \bt u$, $w^2 = (\bt+1)w-\bt u$ \\
  $(a_i, a_i) = (u,u) = 1$, $(a_i, a_{i+1}) = (u,w) = \bt$\\
  $(a_i, a_{i+2}) = \frac{\bt}{2}$, $(a_i, a_{i+3}) = (a_i, u) = \frac{\al}{2}$, \\
  $(a_i, w) = \al$, $(w,w) = \bt +2$
  \vspace{4pt}
  \end{tabular}
\\
$6\Y(\frac{1}{2}, 2)$ & \begin{tabular}[t]{c} $a_0, a_2, a_4, d, z$ \\ \mbox{where} \\ $a_i := a_{i+3}+d$ \end{tabular} &
  \begin{tabular}[t]{c}
  $a_i a_{i+2} = (a_i + a_{i+2}-a_{i+4})$ \\
  $a_i d = \frac{1}{2}d +z $ \\
  $d^2 = -2 z$, $za_i = zd = z^2= 0$ \\
  $(a_i, a_i) = 1$, $(a_i, a_j) = 1$\\
  $(d, x) = (z,x) = 0$, for all $x \in A$
  \vspace{4pt}
  \end{tabular}
 
\end{tabular}
\caption{Symmetric $2$-generated $\cM(\al,\bt)$-axial algebras on $X(6)$}\label{tab:multX6}
\end{table}

\begin{lemma}
The algebras $6\A(\al, \frac{-\al^2}{4(2\al-1)})$ and $6\J(2\bt, \bt)$ have property \textup{(J)} with the additional Jordan type axis being $c$ and $u$, respectively \textup{(}see Table $\ref{tab:multX6}$ on page $\pageref{tab:multX6}$\textup{)}.
\end{lemma}
\begin{proof}
We just show the proof for $6\A(\al, \frac{-\al^2}{4(2\al-1)})$.  The proof for $6\J(2\bt, \bt)$ is similar.  First, observe from the multiplication defined in Table \ref{tab:multX6}, that $c$ is the third axis in each $\la a_i, a_{i+3}, c \ra \cong 3\C(\al)$.  Since additionally $cz = 0$, by counting dimensions, we see that $c$ is semisimple with eigenvalues $1$, $0$ and $\al$ and $A_1(c) = \la c \ra$, $A_0(c) = \la z, a_i +a_{i+3} - \al c : i = 1,2,3\ra$ and $A_\al(c) = \la a_i-a_{i+3} : i = 1,2,3 \ra$.

We will now compute the fusion law for $c$.  Define $\sg_c \colon A \to A$ to be the linear map that fixes $c$ and $z$ and switches $a_i$ with $a_{i+3}$, for $i = 1,2,3$.  One can easily check that $\sg_c$ preserves the algebra product and hence is an automorphism of $A$.  Now, observe that $\sg_c$ fixes $A_+ := A_1(c) \oplus A_0(c)$ and negates $A_- := A_\al(c)$.  Since $\sg_c$ is an involutory automorphism, we have $A_- A_- \subseteq A_+$, $A_+ A_- \subseteq A_-$ and $A_+A_+ \subseteq A_+$.  Clearly we have $A_1(c)A_1(c) = A_1(c)$ and $A_1(c) A_0(c) = 0$, so it remains to show that $A_0(c) A_0(c) \subseteq A_0(c)$.  For this we calculate.  We have $z^2 \in \la z \ra \subseteq A_0(c)$ and, since $\la a_i, a_{i+3}, c \ra \cong 3\C(\al)$, $(a_i+a_{i+3}-\al c)^2 \in A_0(c)$.  We have 
\begin{align*}
z(a_i+a_{i+3}-\al c) &= \tfrac{\al(3\al-2)}{4(2\al-1)}\big( 2a_i -a_{i-2}-a_{i+2} +z \\
&\phantom{{}=\tfrac{\al(3\al-2)}{4(2\al-1)}\big({}} {} + 2a_{i+3} -a_{i+1}-a_{i+5} +z +0 \big) \\
&=\tfrac{\al(3\al-2)}{4(2\al-1)}\big( 2z + 2(a_i+a_{i+3} -\al c) \\
&\phantom{{}=\tfrac{\al(3\al-2)}{4(2\al-1)}\big({}} {} -(a_{i+1} + a_{i+4} - \al c) - (a_{i+2} + a_{i+5} - \al c) \big)
\end{align*}
which is in $A_0(c)$.  Another straightforward, but slightly longer computation gives
\begin{align*}
(a_i+a_{i+3} -\al c)(a_{i+1}+a_{i+4}-\al c) &= -\tfrac{\al(5\al-1)}{4(2\al-1)} z + \tfrac{\al}{4}(a_i +a_{i+3} - \al c) \\
&\phantom{{}={}} + \tfrac{\al}{4}(a_{i+1} +a_{i+4} - \al c) \\
&\phantom{{}={}} + \tfrac{\al(4\al-1)}{4(2\al-1)}(a_{i+2} + a_{i+5} - \al c)
\end{align*}
and so $A_0(c)A_0(c) \subseteq A_0(c)$ as required.
\end{proof}

\begin{lemma}
$6\Y(\frac{1}{2}, 2)$ does not have property \textup{(J)}.
\end{lemma}
\begin{proof}
Suppose, for a contradiction that it did have property (J) with respect to some axis $a$.  Then, by Lemma \ref{propertyJlem}, $a$ would be in each Jordan type $\al$ subalgebra $\lla a_i, a_{i+3} \rra$ for $i = 1,2,3$.  However, $a_{i+3} = a_i + d$ (see Table \ref{tab:multX6}) and $\lla a_i, a_{i+3} \rra = \la a_i, a_i+d, z \ra \cong \Cl$.  Looking again at Table \ref{tab:multX6}, $\lla a_0, a_3 \rra \cap \lla a_1, a_4 \rra \cap \lla a_2, a_5 \rra = \la d, z\ra$.  So, $a = \gm d + \dl z$ for some $\gm, \dl \in \FF$.  We have 
\[
a = a^2 = (\gm d + \dl z)^2 = \gm^2 d^2 = -2\gm^2z
\]
a contradiction since $a$ is idempotent and $z^2 = 0$.
\end{proof}


\subsection{Finite dimensional algebras on $X(\infty)$}\label{sec:3Y5Y}

There are two families of finite-dimensional algebra on Yabe's list that generically have the axet $X(\infty)$, namely $\IY_3(\al,\frac{1}{2}, \mu)$\footnote{Note that the parameter $\mu$ we use is different to Yabe's parameter.  Here $\IY_3(\al,\frac{1}{2}, \mu) \cong \mathrm{III}(\al, \frac{1}{2}, -2\mu-1)$.} and $\IY_5(\al, \frac{1}{2})$.  Since $\bt=\frac{1}{2}$ for both these algebras, we assume that $\ch(\FF) \neq 2$.  We must decide when these algebras have finite axets and when they have property (J).  We again use our own bases for these algebras, which will allow us to better exhibit their axial structure.

For $\IY_3(\al,\frac{1}{2}, \mu)$, we note that $\al \neq \bt = \frac{1}{2}$ (as well as $\al \neq 1,0$).  We split into three subcases: $\al = -1$, $\mu = 1$ and otherwise.  The last case is dealt with by considering split spin factor algebras.

\begin{definition}\cite{splitspin}
Let $E$ be a vector space with a symmetric bilinear form $b$ and $\al \in \FF$.  The \emph{split spin factor} algebra $S(b, \al)$ is the algebra on $E \oplus \FF z_1 \oplus \FF z_2$ with multiplication
\[
\begin{gathered}
z_1^2 = z_1, \quad z_2^2 = z_2, \quad z_1 z_2 = 0, \\
e z_1 = \al e, \quad e z_2 = (1-\al) e, \\
ef = -b(e,f)z,
\end{gathered}
\]
for all $e,f \in E$, where $z := \al(\al-2) z_1 + (\al-1)(\al+1)z_2$.
\end{definition}

From \cite{splitspin}, we have the following properties.  If $e \in E$ and $b(e,e) =1$, then $x := \frac{1}{2}(e + \al z_1 + (\al+1)z_2)$ is a (primitive) axis of Monster type $\cM(\al, \frac{1}{2})$ and these are all such axes (not counting $z_1$, which is of Jordan type $\al$).  Therefore, $S(b,\al)$ is an axial algebra of Monster type $\cM(\al, \frac{1}{2})$ if and only if $E$ is spanned by vectors of norm $1$.  Since we are interested in $2$-generated algebras, we take $E = \la e,f \ra$ to be $2$-dimensional.  Let $b(e,e) = b(f,f) = 1$ and $b(e,f)=\mu$, for some $\mu \in \FF$.  Define $x := \frac{1}{2}(e + \al z_1 + (\al+1)z_2)$ and $y := \frac{1}{2}(f + \al z_1 + (\al+1)z_2)$.

\begin{theorem}\textup{\cite[Theorem $5.1$]{splitspin}}
If $\al \neq -1$ and $\mu \neq 1$, then $\lla x,y \rra = S(b,\al) \cong \IY_3(\al, \frac{1}{2}, \mu)$.
\end{theorem}

It is easy to see that $S(b, \al)$ has an identity given by $\1 = z_1+z_2$, so in particular, $\IY_3(\al, \frac{1}{2}, \mu)$ has an identity whenever $\al\neq -1$ and $\mu \neq 1$.  It turns out that when $\al = -1$, or $\mu = 1$, the algebra has no identity and instead has a nil element.  Also, in these two cases, $\lla x,y \rra$ is a proper subalgebra.  Indeed for $\al=-1$ and $\mu \neq 1$, $S(b,-1)^\circ := \lla x,y \rra$ has codimension $1$.  In \cite{splitspin}, we introduce a nil cover of this subalgebra.

\begin{definition}
Let $E$ be a vector space with a symmetric bilinear form $b$.  Let $\widehat{S}(b,-1)^\circ$ be the algebra on $E \oplus \FF z_1 \oplus \FF n$ with multiplication
\[
\begin{gathered}
z_1^2 = z_1, \quad n^2 = 0, \quad z_1 n = 0, \\
e z_1 =  -e, \quad e n = 0, \\
ef = -b(e,f)z,
\end{gathered}
\]
for all $e,f \in E$, where $z := 3 z_1 -2n$.
\end{definition}

The analogous results to the above also hold for $\widehat{S}(b,-1)^\circ$.  That is, if $b(e,e) =1$ then $x = \frac{1}{2}(e-z_1+n)$ is an axis of Monster type $\cM(-1,\frac{1}{2})$ and these are all such axes.  Hence, $\widehat{S}(b,-1)^\circ$ is an axial algebra of Monster type $\cM(-1,\frac{1}{2})$ if and only if $E$ is spanned by vectors of norm $1$.  Suppose $E = \la e, f \ra$ is as above with $b(e,e) = b(f,f) = 1$ and $b(e,f)=\mu$, $x := \frac{1}{2}(e - z_1 + n)$ and $y := \frac{1}{2}(f -z_1 + n)$.  By \cite[Theorem 6.9]{splitspin}, if $\mu \neq 1$, then $\lla x,y \rra = \widehat{S}(b,-1)^\circ \cong \IY_3(-1, \frac{1}{2}, \mu)$.

We may now consider the axets for both of these algebras simultaneously.  From \cite{splitspin}, the Miyamoto involution for $x$ is $\tau_x = -r_e$ (cf. Section \ref{sec:jordanaxet}).  Let $\theta$ be the flip automorphism which switches $x$ and $y$.  Note that it fixes $z_1$ and $z_2$ (or $n$) and so acts on $E$ by switching $e$ and $f$.  Define $\rho := \theta \tau_x$.  Recall from Subsection \ref{sec:jordanaxet}, the definition of $N(n)$ (note that our form here is scaled by a factor of $\frac{1}{2}$ compared to the form in Subsection \ref{sec:jordanaxet}; equivalently, $N(n)$ defined in this paper is scaled by a factor of 2 compared to \cite{splitspin}).

\begin{lemma}\textup{\cite[Lemma 5.3]{splitspin}}
$\IY_3(\al, \frac{1}{2}, \mu)$, $\mu \neq 1$, has axet $X(n)$, where $2\mu \in N(n)$.
\end{lemma}

We can now see that $\IY_3(\al, \frac{1}{2}, \mu)$, $\mu \neq 1$, has property (J).

\begin{lemma}
Let $n \geq 3$ odd and $2\mu \in N(2n)$.  Then $A = \IY_3(\al, \frac{1}{2}, \mu)$ has property \textup{(J)} with respect to $z_1$.
\end{lemma}
\begin{proof}
Since $x$ and $x^-$ have the same Miyamoto involution and are the only such axes, if $x \in X$ then $x^-$ is also in $X$ and moreover $x$ and $x^-$ lie in different orbits on $X \cong X(2n)$.  From \cite{splitspin}, $z_1$ is an axis of Jordan type $\al$.  Since $A_\al(z_1) = E$, $\tau_{z_1}$ maps $e$ to $-e$ and so maps $x$ to $x^-$, thus switching the two orbits.
\end{proof}

When $\mu = 1$, $\IY_3(\al, \frac{1}{2},1)$ has a different structure.  We define it with basis\footnote{Compared to Yabe's basis, we have $z := \frac{1}{2}(a_1+a_{-1}) -a_0 +\frac{1}{\al}q$ and $n := \frac{1}{2\al}q$.} $a_0, a_1, z, n$ with the following multiplication:
\[
\begin{gathered}
a_i^2 = a_i, \qquad a_0 a_1 = \tfrac{1}{2}(a_0+a_1) +(\al-\tfrac{1}{2})z +n, \\
 z^2 = 0 = n^2, \qquad a_iz = \al z, \qquad a_i n = 0 = zn.
\end{gathered}
\]
In $\IY_3(\al, \frac{1}{2},1)$, all axes have a common $\al$ eigenspace $\la z \ra$ and a common $0$-eigenspace $\la n \ra$.

\begin{lemma}
$\IY_3(\al, \frac{1}{2}, 1)$ has axet $X(\infty)$ if $\ch(\FF)=0$ and axet $X(p)$ if $\ch(\FF)=p$.
\end{lemma}
\begin{proof}
An easy calculation shows that any idempotent in $A = \IY_3(\al, \frac{1}{2}, 1)$ has the form $\lm a_0 + (1-\lm)a_1 + \lm(\lm-1)(2\al-1)z + 2\lm(1-\lm)n$, for $\lm \in \FF$.  In particular, the difference of two axes never lies in the ideal $\la z \ra$.  So axes in $A$ are in bijection with axes in the quotient $\bar{A} = A/\la z \ra$.  However, $\bar{A}$ is a $2$-generated axial algebra of Jordan type $\frac{1}{2}$ containing a nilpotent ideal $\la \bar{n} \ra$.  By Theorem \ref{Jordanclassification} and Proposition \ref{spinsimple}, the only possibilities are $S(-2)$ and $\Cl$.  However, $\bar n$ annihilates every element of $\bar{A}$ and so $\bar{A} \cong \Cl$.  Hence axet for $A$ is the same as the axet for $\bar{A} \cong \Cl$ and so the result follows from Lemma \ref{Cl00axet}.
\end{proof}

In particular, the axet for $\IY_3(\al, \frac{1}{2}, 1)$ is never even in size and so it cannot have property (J).

We now turn to the algebra $\IY_5(\al, \frac{1}{2})$.  From \cite{yabe}, this algebra has axial dimension $5$.  That is, the subspace spanned by the axes has dimension $5$ and in fact we have the relation $a_{i+5} -5a_{i+4} +10a_{i+3} -10a_{i+2} +5a_{i+1} -a_i =0$, for $i \in \Z$.  (Recall the notation $a_i$ from the beginning of Section \ref{sec:2genalbt}.)  Moreover, any five consecutive axes are a basis for this subspace and so we may express $a_n$ as a linear combination of, say, $a_0, \dots, a_4$.

Before we prove this, we need a technical lemma.

\begin{lemma}\label{techpoly}
If $P(t) \in \FF[t]$ is a polynomial of degree strictly less than $k$, then $\sum_{t =0}^k (-1)^t\binom{k}{t}P(t) = 0$.
\end{lemma}
\begin{proof}
Since the polynomials $Q_s(t) := t(t-1)\dots (t-s+1)$, for $0 \leq s <k$, form a basis for the space of polynomials of degree at most $k$, it suffices to prove this for each $Q_s(t)$.  Consider $(1+x)^k = \sum_{t=0}^k {k \choose t} x^t$ and differentiate both sides $0 \leq s < k$ times with respect to $x$.  By setting $x=-1$, we obtain $0 = \sum_{t =0}^k (-1)^{t-s} {k \choose t} t(t-1) \dots (t-s+1) = (-1)^{-s} \sum_{t =0}^k (-1)^t {k \choose t} Q_s(t)$ and hence the claim is shown.
\end{proof}

We can now express $a_n$ as a linear combination of $a_0, \dots, a_4$.

\begin{lemma}\label{5Yaxes}
In $\IY_5(\al, \frac{1}{2})$, for $n \in \Z$ and $\ch(\FF) \neq 3$, we have
\[
a_n =  \sum_{i=0}^4  \frac{(-1)^i }{i! (4-i)!} \frac{n^{\underline{5}}}{(n-i)} a_i
\]
where $n^{\underline{5}} := n(n-1)\dots(n-4)$.\footnote{Note that we cancel $n-i$ in the denominator with a term in $n^{\underline{5}}$ before evaluating this formula and so every term makes sense provided $\ch(\FF) \neq 3$.}
\end{lemma}
\begin{proof}
We prove this by induction on $n$.  One can easily check that this holds for the base cases $n = 0, \dots, 4$.  We will write $f_i(n) := \frac{(-1)^i }{i! (4-i)!} \frac{n^{\underline{5}}}{(n-i)}$ so that we need to prove that $a_n = \sum_{i = 0}^4 f_i(n) a_i$.  Using the recurrence relation for $a_{n+1}$ we have
\begin{align*}
a_{n+1} &= 5a_{n} - 10 a_{n-1} + 10a_{n-2} -5 a_{n-3} + a_{n-4} \\
&= \sum_{j=0}^4 (-1)^j{5 \choose j} a_{n-4+j} \\
&= \sum_{j=0}^4 (-1)^j{5 \choose j} \sum_{i = 0}^4 f_i(n-4+j) a_i\\
&= \sum_{i = 0}^4 \left(\sum_{j=0}^4 (-1)^j{5 \choose j} f_i(n-4+j) \right)a_i
\end{align*}
Now observe that all of our $f_i(n-4+j)$ are polynomials in $j$ of degree strictly less than $5$.  By Lemma \ref{techpoly}, we have that 
\[
\sum_{j=0}^4 (-1)^j{5 \choose j} f_i(n-4+j) = (-1)^5 {5 \choose 5} f_i(n-4+5) = f_i(n+1)
\]
and hence the result follows by induction.
\end{proof}

\begin{corollary}\label{5Yaxet}
The algebra $\IY_5(\al, \frac{1}{2})$ has axet $X(p)$ over a field of characteristic $p \geq 5$, $X(9)$ over a field of characteristic $3$ and axet $X(\infty)$ otherwise.
\end{corollary}
\begin{proof}
Since $A = \IY_5(\al, \frac{1}{2})$ is symmetric, it has a regular axet.  So $A$ has finite axet $X(n)$ if and only if $n > 0$ is minimal such that $a_n = a_0$ (see the discussion of the notation $a_i$ at the beginning of this section).  If $\ch(\FF) \neq 3$, by Lemma \ref{5Yaxes}, we can express $a_n$ in the basis $a_0, \dots, a_4$ as $a_n =  \sum_{i=0}^4  \frac{(-1)^i }{i! (4-i)!} \frac{n^{\underline{5}}}{(n-i)} a_i$.  Hence $a_0 = a_n$ if and only if $1 = \frac{1}{4!}(n-1)(n-2)(n-3)(n-4)$ and $0 = \frac{(-1)^i}{i!(4-1)!} \frac{n^{\underline{5}}}{(n-i)}$ for $i = 1, \dots, 4$.  First observe that the latter equations are never simultaneously satisfied when $\ch(\FF) = 0$.  So we now assume that $\ch(\FF) = p \geq 5$.  From the first equation we see that $p$ does not divide $(n-1)(n-2)(n-3)(n-4)$.  Hence from the latter equations, we see that $p|n$ is the only solution.  Since $n$ is the minimal such integer, we have $n=p$ and axet $X(p)$ in positive characteristic $p > 3$.

For $\ch(\FF)=3$, using the relation $r_i := a_{i+5} -5a_{i+4} +10a_{i+3} -10a_{i+2} +5a_{i+1} -a_i$ we see by inspection that $a_k \neq 0$, for $k = 5, \dots, 8$.  Whereas $0 = r_4 + 5r_3 + 15r_2 + 5r_1 + r_0 = a_9 -30 a_6 +81 a_5 -81 a_4 +30 a_3 -a_0 = a_9 -a_0$ and hence we get axet $X(9)$.
\end{proof}

In particular, $\IY_5(\al, \frac{1}{2})$ never has an axet of even size and hence it never has property (J).

\subsection{The Highwater algebra and its characteristic 5 cover}\label{sec:HW}

The Highwater algebra $\cH$ was introduced by Franchi, Mainardis and Shpectorov in \cite{highwater} and also discovered independently by Yabe in \cite{yabe}.  It is an infinite dimensional $2$-generated symmetric $\cM(2, \frac{1}{2})$-axial algebra over any field of characteristic not $2$, or $3$.  Its axet is $X(\infty)$.  In characteristic $5$, the Highwater algebra has a cover $\hatH$, introduced by Franchi and Mainardis \cite{highwater5}, which is also an $\cM(2, \frac{1}{2})$-axial algebra.

The algebras $\cH$ and $\hatH$ have many ideals and hence many quotients, including ones with a finite axet.  In this subsection, we consider two questions:
\begin{enumerate}
\item which quotients of $\cH$ and $\hatH$ have a finite odd axet $X(n)$;
\item which quotients of $\cH$ and $\hatH$ have axet $X(2n)$, with $n$ odd, and property (J).
\end{enumerate}

In \cite{HWquo}, Franchi, Mainardis and M\textsuperscript{c}Inroy classify the ideals of $\cH$ and $\hatH$.  In particular, they show that, for every $n \in \N$, there is a universal (largest) quotient $\cH_n$ of $\cH$ (and $\hatH_n$ of $\hatH$) which has axet $X(n)$.  The quotient $\cH_n$ has dimension $n + \lfloor \frac{n}{2} \rfloor$.  In characteristic $5$, $\hatH_n$ coincides with $\cH_n$ unless $3|n$, in which case it has dimension $n + \lfloor \frac{n}{2} \rfloor + 2 \lfloor \frac{n}{6} \rfloor$.  This answers the first question.  In the remainder of the section we focus on the second question and show the following.

\begin{proposition}\label{HWpropertyJ}
Suppose $n \in \N$ is odd.
\begin{enumerate}
\item There is a universal quotient $\cH_{2n}^J$ of $\cH_{2n}$ which has axet $X(2n)$ and property \textup{(}J\textup{)}.  It has dimension $n+\lceil \frac{n}{2} \rceil+1$.
\item Similarly, in characteristic $5$, there exists a universal quotient $\hatH_{2n}^J$ of $\hatH_{2n}$ which has axet $X(2n)$ and property \textup{(}J\textup{)}.  Furthermore, $\hatH_{2n}^J = \cH_{2n}^J$ if $3 \nmid n$ and it has dimension $n+\lceil \frac{n}{2} \rceil+ 2\lceil \frac{n}{6} \rceil + 1$ if $3|n$.
\end{enumerate}
\end{proposition}

In \cite{HWquo}, Franchi, Mainardis and M\textsuperscript{c}Inroy define a cover $\hatH$ for $\cH$ in all characteristics.  It contains an ideal $J$ such that $\hatH/J$ is isomorphic to the Highwater algebra.  In characteristic $5$, $\hatH$ coincides with Franchi and Mainardis's cover of the Highwater algebra.  In other characteristics, $\hatH$ is a symmetric $2$-generated axial algebra, but for a larger fusion law.  However, it can be used to prove statements about both $\cH$ and $\hatH$ in a unified way.  This is precisely the approach taken by Franchi, Mainardis and M\textsuperscript{c}Inroy in \cite{HWquo}.  We will follow a similar approach here, giving a unified description of the quotients which have property (J).  Hence we will be working in a factor algebra of $\hatH_{2n}$ which is an algebra of Monster type $(2, \frac{1}{2})$.

For $r \in \Z$, we write $\r$ for its image in $\Z_3 = \Z/3\Z$.

\begin{definition}\cite{HWquo}
Let $\hatH$ be the algebra over a field $\FF$, where $\ch(\FF) \neq 2,3$, with basis
\[
\{ a_i : i \in \Z \} \cup \{ s_j : j \in \N \} \cup \{ p_{\r,k} : \r \in \{ \overline{1}, \2 \}, k \in 3\N \}
\]
Set $s_0 = 0$, $p_{\r,j} = 0$ for all $\r \in \Z_3$, if $j \notin 3\N$, $p_{\0,j} := -p_{\overline{1},j} - p_{\2,j}$ and $z_{\r,j} = p_{\r+\overline{1},j} - p_{\r-\overline{1},j}$.  We define multiplication on the basis of $\hatH$ by
\begin{enumerate}
\item $a_i a_j := \frac{1}{2}(a_i+a_j) +s_{|i-j|}+z_{\ii,|i-j|}$
\item $a_i  s_{j} := -\frac{3}{4} a_i + \frac{3}{8}( a_{i-j}+ a_{i+j}) +\frac{3}{2} s_{j} -z_{\ii,j}$
\item $a_i p_{\r,j}:=\frac{3}{2}p_{\r,j} - p_{-(\ii+\r),j} $
\item $s_j  s_ l:= \frac{3}{4}( s_{j}+ s_{l}) - \frac{3}{8}(s_{|j-l|} + s_{j+l})$
\item $s_j p_{\r,k}:=\frac{3}{4}( p_{\r, j}+ p_{\r, k}) - \frac{3}{8}(p_{\r, |j-k|} + p_{\r, j+k})$
\item $
 p_{\r,h} p_{\t,k}:= \frac{1}{4}(z_{-(\r+\t), h}+ z_{-(\r+\t),  k})
- \frac{1}{8}(z_{-(\r+\t), |h-k|}+z_{-(\r+\t), h+k})
 $
\end{enumerate}
where $i \in \Z$, $j,l \in \N$, $h,k \in 3\N$ and $\r,\t \in \Z_3$.
\end{definition}

Naturally, the $a_i$ are the axes of this algebra and they form the axet $X(\infty)$.  Note that, $J = \la  p_{\r,j} : \r \in \Z_3, j \in 3\N \ra$ and $\cH \cong \hatH/J$.  So the images of the $p_{\r,j}$ and $z_{\r,j}$ are all $0$ in $\cH$, which greatly simplifies the above definition.

From \cite{HWquo}, $\Aut(\hatH) \cong D_\infty$.  Indeed, for $k \in \frac{1}{2}\Z$, let $\tau_k$ be the reflection given by $i \mapsto 2k-i$ and $D = \la \tau_0, \tau_{\frac{1}{2}} \ra$.  Let $\sgn \colon D \to \Z$ be the sign representation of $D$ and so $\sgn(\rho) = -1$ if $\rho$ is a reflection and $\sgn(\rho)=1$ if $\rho$ is a translation.  Then, for $g \in D = \Aut(\hatH)$,
\[
{a_i}^g = a_{i^g}, \quad {s_j}^g = s_j, \quad {p_{\r,k}}^g = \sgn(g) p_{\overline{r^g},k}
\]
Moreover, the Miyamoto involution $\tau_{a_j}$ coincides with $\tau_j$ and the flip automorphism $f$ is equal to $\tau_{\frac{1}{2}}$.

\begin{theorem}\textup{\cite[Corollary 10.1]{HWquo}}\label{HWideals}
For $n \in N$, let $I_n \unlhd \hatH$ be the ideal generated by $a_0-a_n$ and define $\hatH_n := \hatH/I_n$ and $\cH_n = \hatH/JI_n$.  
Moreover, let $B$ be the set of the following elements:
\begin{align*}
& a_i - a_{i+n}, & \mbox{for } i \in \Z, \\
& s_{j} - s_{j+n}, \  s_{jn}, & \mbox{for } j \in \N, \\
& s_{j} - s_{n-j}, & \mbox{for } 1 \leq j \leq \left\lfloor \tfrac{n}{2} \right\rfloor.
\end{align*}

\begin{enumerate}
\item $\cH_n$ and $\hatH_n$ are $2$-generated symmetric axial algebras with axet $X(n)$.  Moreover, every quotient of $\cH$, respectively $\hatH$, with axet $X(n)$ is a quotient of $\cH_n$, respectively $\hatH_n$.
\item If $3 \nmid n$, then
\begin{enumerate}
\item $I_n$ is spanned by the union of $B$ and a basis for $J$.
\item $\hatH_n = \cH_n$ has a basis given by the images of $\{ a_i : 0 \leq i <n \} \cup \{ s_j : 1 \leq j \leq \lfloor \frac{n}{2}\rfloor \}$.
\end{enumerate}
\item If $3|n$, then 
\begin{enumerate}
\item $I_n$ is spanned by the union of $B$ and the set of elements
\begin{align*}
& p_{\r, j} - p_{\r,j+n}, \  p_{\r, jn}, & \mbox{for } j \in \N, r=1,2, \\
& p_{\r,j} - p_{\r,n-j}, & \mbox{for } 1 \leq j \leq \left\lfloor \tfrac{n}{2} \right\rfloor,  r=1,2.
\end{align*}
\item $\hatH_n$ has a basis given by the images of $\{ a_i : 0 \leq i <n \} \cup \{ s_j : 1 \leq j \leq \lfloor \frac{n}{2}\rfloor \} \cup \{ p_{\overline{1},j}, p_{\2,j} : j \in 3\Z, 1 \leq j \leq \lfloor \frac{n}{2}\rfloor \}$.
\end{enumerate}
\end{enumerate}
\end{theorem}

By \cite[Theorem 1.4]{HWquo}, every ideal of $\hatH$ is invariant under the full automorphism group and hence every quotient is symmetric and has a regular axet.   In particular, the automorphism group acts transitively on the axet $X(n)$ in $\cH_n$ and $\hatH_n$ for all $n \in \N$.

Recall that we wish to find the quotients of $\cH_{2n}$ (or $\hatH_{2n}$ when $\ch(\FF) = 5$), for $n$ odd, which have property (J).  For simplicity, we use the same notation $a_i$ and $s_j$ for elements of $\cH$ and their images in $\cH_n$; similarly for $\hatH$ and $\hatH_n$.  As before, the opposite axes $a_i$ and $a_{i+n}$ in $\cH_{2n}$ (or $\hatH_{2n}$) generate a $2$-generated axial algebra of Jordan type $2$.  Since $\ch(\FF) \neq 3$, $2 \neq \frac{1}{2}$ and so by Theorem \ref{Jordanclassification}, $\lla a_i, a_{i+n} \rra \cong 3\C(2)$.  By Lemma \ref{propertyJlem}, any additional axis $a$ of Jordan type $2$ giving the algebra property (J) must lie in the intersection of all the $3\C(2)$ subalgebras $\lla a_i, a_{i+n} \rra$.  In particular, $a$ is an axis of Jordan type $2$ in each subalgebra.  The algebra $3\C(2)$ only has 6 primitive idempotents, three of Jordan type $2$ and three of Jordan type $1-2 = -1$.  Since $\ch(\FF) \neq 3$, $-1 \neq 2$ and so $a$ must be the third axis of Jordan type $2$ in each $\lla a_i, a_{i+n} \rra$ subalgebra.
 
\begin{lemma}\label{HW3C}
In $\hatH_{2n}$, for every $i \in 0, \dots, n-1$, the third axis of the subalgebra $\lla a_i, a_{i+n} \rra = \la a_i, a_{i+n}, s_n + z_{\ii,n} \ra \cong 3\C(2)$ is $b_i := \frac{1}{2}(a_i + a_{i+n}) - (s_n + z_{\ii,n})$.
\end{lemma}
\begin{proof}
The third axis in a $3\C(2)$ subalgebra generated by axes $x$ and $y$ is given by $x+y - xy$.  For $\hatH_{2n}$, this is $b_i := a_i + a_{i+n} - a_i a_{i+n} = a_i + a_{i+n} -\frac{1}{2}(a_i + a_{i+n}) - (s_n + z_{\ii,n}) = \frac{1}{2}(a_i + a_{i+n}) - (s_n + z_{\ii,n})$.\footnote{When $\hatH_{2n}$ is not of Monster type, one can check that $\lla a_i, a_{i+n} \rra$ is still isomorphic to a $3\C(2)$ subalgebra.  This is automatic if it is of Monster type.}
\end{proof}

By Theorem \ref{HWideals} (2) and (4), $b_i \neq b_j$ for $0 \leq i \neq j \leq n-1$.  However, we may quotient by an ideal to force these to be equal.  The smallest such ideal $I$ is generated by $b_i-b_j$ for all $0 \leq i < j \leq n-1$.  Note that, since ideals are invariant under the action of the Miyamoto group, $I$ is in fact generated by $b_0-b_1$.

Before calculating an explicit basis for the ideal $I$, we need a technical lemma from \cite{HWquo}, which follows immediately from the multiplication in $\hatH$.

\begin{lemma}\textup{\cite[Lemma 3.8]{HWquo}}\label{HWquolem}
For $i\in \Z$, $j \in \N$, $h,k \in 3\N$ and $\{ \r, \t\} \subseteq \Z_3$, we have the following.
\begin{enumerate}
\item $a_i z_{\r,j}=\frac{3}{2}z_{\r,j} + z_{-(\ii+\r),j}$
\item $s_j z_{\r,k}=\frac{3}{4}( z_{\r,j}+ z_{\r, k}) - \frac{3}{8}(z_{\r, |j-k|} + z_{\r, j+k})$
\item $p_{\r,h}z_{\t,k}=\frac{3}{4}( p_{-(\r+\t),h}+ p_{-(\r+\t),k}) - \frac{3}{8}(p_{-(\r+\t), |h-k|} + p_{-(\r+\t), h+k})$
\item $z_{\r,h}z_{\t,k}= -\frac{3}{4}( z_{-(\r+\t),h}+ z_{-(\r+\t),k}) + \frac{3}{8}(z_{-(\r+\t), |h-k|} + z_{-(\r+\t), h+k})$
\end{enumerate}
\end{lemma}

\begin{lemma}\label{HWJ}
Let $I = (b_0-b_1) \unlhd \hatH_{2n}$.  If $J \subset I_{2n}$, then $I$ has a basis given by
\begin{align*}
& a_0+a_{n} - (a_{i} + a_{i+n}) -2(z_{\0,n} - z_{\ii,n}), &1 \leq i \leq n-1,\\
& s_j + s_{n-j} - s_n, & 1 \leq j \leq \left\lfloor \tfrac{n}{2} \right\rfloor.\\
\intertext{If $J \not \subset I_{2n}$, then $3|n$ \textup{(}and $\ch(\FF) = 5$\textup{)}.  Then $I$ has a basis given by the above elements and}
& p_{\r,j} + p_{\r,n-j} - p_{\r,n}, & j \in 3\N, 1 \leq j \leq \left\lfloor \tfrac{n}{2} \right\rfloor, r = 1,2.
\end{align*}
\end{lemma}
\begin{proof}
First, $2(b_i-b_j) \in I$ and these are linear combinations of $2(b_0-b_i) = a_0 + a_n - a_i - a_{i+n}-2(z_{\0,n} - z_{\ii,n})$, where $i = 1, \dots, n-1$.  We wish to calculate $a_k(b_i-b_j)$, for all $0 \leq k <2n$.  Since the automorphism group acts transitively on the axet, it also acts transitively on the set of $b_i$.  Let $g \in \Aut(\hatH_{2n})$ such that $a_k^g = a_0$.  Then, $a_k(b_i-b_j) = (a_0(b_{i^g}-b_{j^g}))^{g^{-1}}$.  Since the $b_0-b_i$ span the space of differences of the $b_i$, it suffices to calculate just $a_0(b_0 - b_i)$ and take the orbit under the automorphism group.

Now, using Lemma \ref{HWquolem} (1), we calculate
\begin{align*}
2 a_0(b_0-b_i) &= a_0\big(a_0 + a_n - a_i - a_{i+n}-2(z_{\0,n} - z_{\ii,n})\big) \\
 &= \tfrac{1}{2}(a_0 + a_n - a_i - a_{i+n}) +s_{n} -s_{i} -s_{i+n} \\
 &\phantom{{}={}} + z_{\0,n} -z_{\0,i} -z_{\0,i+n} - 3(z_{\0,n} - z_{\ii,n}) -2(z_{\0,n} - z_{\ii,n}) \\
 &= \tfrac{1}{2}(b_0-b_i) + s_{n} -s_{i} -s_{i+n} + z_{\0,n} -z_{\0,i} -z_{\0,i+n} \\
 &\phantom{{}={}} -2( 2z_{\0,n} - (z_{\ii,n} + z_{-\ii,n}))
\end{align*}
First, suppose that $J \subset I_{2n}$.  Then all the $z_{\r,j} \in J \subset I_{2n}$ and hence the above yields $s_{n} -s_{i} -s_{i+n} \in I$.  By Theorem \ref{HWideals} (2), $s_{n-i} - s_{n+i} \in I_{2n}$ and so $s_i + s_{n-i} - s_n \in I$ as required.  Since $s_i + s_{n-i} - s_n$ is fixed by $\Aut(\hatH_{2n})$, we do not obtain any additional elements from $a_k(b_i-b_j)$.  Note that it is immediately clear from Theorem \ref{HWideals} (2) that for the $s_{j} +s_{n-j} -s_n$, we may take $1 \leq j \leq \lfloor \frac{n}{2} \rfloor$.

Now suppose that $J \not \subset I_{2n}$; then $3|n$ and by assumption $\ch(\FF) = 5$.  By the above calculation, we have that
\begin{equation}\label{eq:a(b0-b1)}
s_{n} -s_{i} -s_{i+n} + z_{\0,n} -z_{\0,i} -z_{\0,i+n}-2( 2z_{\0,n} - (z_{\ii,n} + z_{-\ii,n})) \in I
\end{equation}

If $i \not \equiv 0 (3)$, then $z_{\0,i} = 0 = z_{\0,n-i}$ and $z_{\0,n} + z_{\ii,n} + z_{-\ii,n} = 0$.  So Equation~(\ref{eq:a(b0-b1)}) reduces to $s_n -s_{i} -s_{i+n} - 5 z_{\0,n} = s_{n} -s_{i} -s_{i+n}\in I$, as $\ch(\FF) = 5$.

If $i \equiv 0 (3)$, then $z_{\ii,n} = z_{-\ii,n} = z_{\0,n}$.  So Equation~(\ref{eq:a(b0-b1)}) reduces to $s_{n} -s_{i} -s_{i+n} + z_{\0,n} -z_{\0,i} -z_{\0,i+n} \in I$.  Since $g \in \Aut(\hatH_{2n})$ fixes all the $s_k$ but acts as ${z_{\r,j}}^g = z_{\overline{r^g},j}$, we get that $s_{n} -s_{i} -s_{i+n} \in I$ and  $z_{\r,n} -z_{\r,i} -z_{\r,i+n} \in I$, for $r = 0,1,2$.  Once again as $s_{n-i} - s_{n+i} \in I_{2n}$, we get $s_i + s_{n-i} - s_n \in I$.  Since $3p_{\r,j} = z_{\r-\overline{1},j} - z_{\r+\overline{1},j}$ and $p_{\r,n-j} - p_{\r,n+j} \in I_{2n}$, we get that $p_{\r,j} + p_{\r,n-j} - p_{\r,n} \in I$ for all $1 \leq j \leq \left\lfloor \tfrac{n}{2} \right\rfloor$ and $r = 1,2$ as required.

A straightforward, but long, calculation now shows that the space spanned by the given elements is closed under multiplication with elements of $\hatH_{2n}$ and so these span the ideal $I$.  By Theorem \ref{HWideals} (2) and (4), they are linearly independent and so they form a basis.
\end{proof}

In particular, $I$ is a proper ideal of $\hatH_{2n}$.  Let $a$ be the common image of all the $b_i = \tfrac{1}{2}(a_i+a_n) -\sg_n$ in $\hatH_{2n}/I$.

\begin{lemma}\label{HWJordan}
The algebra $\hatH_{2n}^J := \hatH_{2n}/I$ has property \textup{(J)} with respect to the axis $a$.
\end{lemma}
\begin{proof}
From the $3\C(2)$ subalgebras, we know that $a_i-a_{i+n}$ is a $2$-eigenvector of $a$ for all $i = 0, \dots, n-1$.  Using the multiplication in $\hatH_{2n}$, we have
\begin{align*}
a s_j &= \left( \tfrac{1}{2}(a_0+a_n) - (s_n + z_{\0,n}) \right) s_j \\
&= -\tfrac{3}{8}(a_0 + a_n) + \tfrac{3}{16}(a_{-j}+a_{n-j} +a_j + a_{j+n}) + \tfrac{3}{2} s_j - \tfrac{1}{2}(z_{\0,j} + z_{\overline{n},j})\\
&\phantom{{}={}} {} - \tfrac{3}{4}(s_n + s_j) + \tfrac{3}{8}(s_{n-j} + s_{n+j})  - \tfrac{3}{4}(z_{\0,n} + z_{\0,j}) + \tfrac{3}{8}(z_{\0,n-j} + z_{\0,n+j}) \\
&= -\tfrac{3}{8}( b_0 - b_j + b_0 - b_{n-j} ) + \tfrac{3}{4}(s_j + s_{n-j} - s_n ) \\
&\phantom{{}={}} {}+ \tfrac{3}{4}(z_{\0,j} + z_{\0,n-j} - z_{\0,n}) -\tfrac{10}{4} z_{\0,j} - \tfrac{3}{8}\big( 2z_{\0,n} - (z_{\jj,n} + z_{-\jj,n}) \big) 
\end{align*}
By Lemma \ref{HWJ}, to show the above is in $I$, we must show that
\begin{equation}\label{HWaJordan}
\tfrac{3}{4}(z_{\0,j} + z_{\0,n-j} - z_{\0,n}) -\tfrac{10}{4} z_{\0,j} - \tfrac{3}{8}\big( 2z_{\0,n} - (z_{\jj,n} + z_{-\jj,n}) \big)
\end{equation}
is in $I$.  If $J \subset I_{2n}$, then all the $z_{\r,k} \in I_{2n}$ and so this is trivially true.  So suppose that $J \not \subset I_{2n}$; then $3|n$ and hence we have that $\ch(\FF) = 5$.

If $j \equiv 0 \mod 3$, then $z_{\jj,n} = z_{-\jj,n} = z_{\0,n}$ and hence, by Lemma \ref{HWJ}, Equation~(\ref{HWaJordan}) reduces to $\tfrac{3}{4}(z_{\0,j} + z_{\0,n-j} - z_{\0,n})$, which is in $I$.

If $j \not \equiv 0 \mod 3$, then $z_{\0,j} = 0 = z_{\0, n-j}$ and $z_{\0,n} + z_{\jj,n} + z_{-\jj,n} = 0$.  So Equation~(\ref{HWaJordan}) reduces to $-\frac{15}{8} z_{\0,n} = 0$ since $\ch(\FF) = 5$.  Therefore in all cases, $s_j$ is a $0$-eigenvector for $a$ in $H_{2n}/I$.  A very similar calculation shows that $p_{\r,j}$ is also a $0$-eigenvector for $a$ in $H_{2n}/I$.

A counting argument shows that $a$, the $a_i- a_{n+i}$, $s_j$ and $p_{\r,j}$ is a basis of eigenvectors for $\hatH_{2n}^J$.  So we see that $a$ is semisimple with eigenvalues $1$, $0$ and $2$.

Since the flip automorphism $f = \tau_{\frac{1}{2}}$ also fixes $I$, it induces an automorphism of $\hatH_{2n}^J$ which we will also call $f$.  Now observe that $(\tau_0 \tau_{\frac{1}{2}})^n$ in $D = \Aut(\hatH)$ maps $i \mapsto i+n$.  So this induces $\rho \in \Aut(\hatH_{2n})$ which fixes $a$ and the $s_j$ and inverts each $a_i - a_{i+n}$.  If $J \subset I_{2n}$, then this is a basis for $\hatH_{2n}^J$.  If $J \not \subset I_{2n}$, then $3|n$ and hence ${p_{\r,j}}^\rho = \sgn(\rho) p_{\overline{r+n}, j} = p_{\r,j}$.  Hence in both cases, $\rho$ is an automorphism which negates the $2$-eigenspace of $a$ and fixes the $1$- and $0$- eigenspaces.  As this preserves the fusion law of $a$, the fusion law must be $\mathbb{Z}_2$-graded.  From the definition, the product of $s_j$ with $s_k$, $s_j$ with $p_{\r,k}$ and $p_{\r,j}$ with $p_{\overline{s},k}$ all lie in the space spanned by the $s_j$ and $p_{\r,k}$.  Hence, for the fusion law, $0 \star 0 \subseteq \{0\}$ and hence $a$ is an axis of Jordan type $2$.
\end{proof}

These lemmas complete the proof of Proposition \ref{HWpropertyJ}.  The claims about the dimensions follows from the bases given in Lemma \ref{HWJ}.  Note that $a$ is not contained in any proper ideal of $\hatH_{2n}^J$.  Indeed, if $a$ were contained in some proper ideal $K$, then $K \cap \lla x, x' \rra$ would be an ideal of $\lla x, x' \rra \cong 3\C(2)$ (and $3\C(2) \not \cong 3\C(\frac{1}{2})$ since $\ch(\FF) \neq 3$ as noted above).  Suppose $K \cap \lla x, x' \rra$ is a proper ideal of $\lla x, x' \rra$.  From Section \ref{sec:Jordan}, as $\ch(\FF) \neq 3$, $3\C(2)$ has a single proper ideal $\la x-x', x-a \ra$ and $a \notin \la x-x', x-a \ra$.  So $K \cap \lla x, x' \rra = \lla x,x' \rra$.  Then, $x, x' \in K$ for all $x \in X$ and hence $I = \hatH$, a contradiction.  Hence $a$ is not contained in any proper ideal and therefore every non-trivial quotient of $\hatH_{2n}^J$ also has property (J).


\section{Collapsing shapes on $3$-generated axets revisited}\label{sec:forbidden}

Having studied the $2$-generated symmetric algebras in the previous section, we can now prove our final theorem.

\begin{theorem} \label{1+nshape}
Let $X = \la a,b,c \ra \cong X_1(1+n)$, $n \geq 3$ odd, be a $3$-generated axet, where $a$ is the fixed axis.  Let $Y = \la a,b\ra$ and $Z = X - \{a\}$.  Suppose that $\Theta$ is a shape on $X$ for $\cM(\alpha, \beta)$ and $A$ is a completion for $\Theta$.
\begin{enumerate}
\item If $\sigma_a = 1$, then $A_Y \cong 2\B$ and $A \cong 1\A \oplus \lla b,c \rra$, where $\lla b,c \rra$ is isomorphic to a quotient of one of:
\begin{enumerate}
\item $3\C(\bt)$, $3\A(\al,\bt)$, $5\A(\al, \frac{5\al-1}{8})$, $\cH_n$, $\hatH_n$,
\item $S(\dl)$, $\IY_3(\al, \frac{1}{2}, \frac{\dl}{2})$ where $\dl \in N(n)$,
\item $\Cl$, $\IY_3(\al, \frac{1}{2}, 1)$, $\IY_5(\al, \frac{1}{2})$ over a field of characteristic $p \geq 3$.
\end{enumerate}
\item If $\sigma_a \neq 1$, then $A_Y \in \cJ^2(\alpha) -\{ 2\B\}$ and either
\begin{enumerate}
\item $A$ is in fact a symmetric $2$-generated algebra and a quotient of one of
\begin{enumerate}
\item $A \cong 6\A(\al, \tfrac{\al^2}{4(1-2\al)})$,
\item $A \cong 6\J(2\bt, \bt)$,
\item $A \cong \IY_3(\al, \tfrac{1}{2}, \mu)$, where $2\mu \in N(2n)$,
\item $A \cong \cH_{2n}^J$
\item $A \cong \hatH_{2n}^J$, where $3|n$, or
\end{enumerate}
\item $(\al,\bt) = (\frac{1}{2}, 2)$ and $A \cong \mathrm{Bar}_{0,1}(\frac{1}{2}, 2)$.
\end{enumerate}
\end{enumerate}
\end{theorem}

\begin{proof}
If $\sg_a=1$, then by Proposition \ref{1+nshapeprop}, $A \cong 1\A \oplus \lla b,c\rra$.  Since $n$ is odd, $\lla b,c \rra$ is a symmetric $2$-generated algebra and so, by the results in Section \ref{sec:2genalbt}, it is one of the algebras listed.

For $\sg_a \neq 1$, we use the case distinction in Proposition \ref{1+nshapeprop}.  Suppose we are in case 2(a).  Then, $A=B$ is a symmetric $2$-generated algebra which has property (J).  Again by the results in Section \ref{sec:2genalbt}, these are the algebras listed.

We now turn to the two exceptional situations 2(b) and 2(c).  The algebra $B$ is a symmetric $2$-generated algebra with axet $X(2n)$.  By Section \ref{sec:2genalbt}, the possibilities are a quotient of one of $S(\dl)$ with $\dl \in N(2n)$, $\cH_{2n}$, $\hatH_{2n}$, $6\A(\al, \tfrac{\al^2}{4(1-2\al)})$, $6\J(2\bt, \bt)$, $6\Y(\frac{1}{2}, 2)$, or $\IY_3(\al, \frac{1}{2}, \mu)$ where $2\mu \in N(2n)$.  Note that $\al = \frac{1}{2}$ and so $B$ cannot also have $\bt = \frac{1}{2}$.  This leaves $6\A(\al, \tfrac{\al^2}{4(1-2\al)})$, $6\J(2\bt,\bt)$, or $6\Y(\frac{1}{2}, 2)$, or one of their quotients.  In these, $\lla b,b'\rra$ is isomorphic to $3\C(\al)$, $3\C(\al)$ and $\Cl$, respectively.  Since $2\B$ is not a quotient of these, case 2(c) where $\lla b,b'\rra \cong 2\B$ does not occur.

For case 2(b), $\lla b, b' \rra \cong S(2)^\circ$, the only choice for $B$ is $6\Y(\frac{1}{2}, 2)^\times$ and so $\bt=2$.  Since $B \cong 6\Y(\frac{1}{2}, 2)^\times$ is $4$-dimensional and spanned by its six axes and $A_Y \cong S(-2)$, we know the multiplication of $a$ with $B$ and hence $A$ is $5$-dimensional.  This is precisely the algebra $\mathrm{Bar}_{0,1}(\frac{1}{2}, 2)$.
\end{proof}

\begin{remark}
The algebra $\mathrm{Bar}_{0,1}(\frac{1}{2}, 2)$ in case 2(b) is one of a family of algebras $\mathrm{Bar}_{i,j}(\frac{1}{2}, 2)$ which are baric \cite{baric}.  That is, there is an algebra homomorphism $w \colon A \to \FF$ called a \emph{weight function}.  In fact, we discovered this family of baric algebras by proving this theorem.  The algebra $\mathrm{Bar}_{0,1}(\frac{1}{2}, 2)$ in the above theorem is $5$-dimensional, cannot be generated by two axes, and contains $6\Y(\frac{1}{2}, 2)^\times \cong \mathrm{Bar}_{1,0}(\frac{1}{2}, 2)$ as a codimension $1$ subalgebra.  We note that we discovered $6\Y(\frac{1}{2}, 2)^\times \cong \mathrm{Bar}_{1,0}(\frac{1}{2}, 2)$ in this way independently to Yabe.  
\end{remark}

\begin{remark}
Note that a $2$-generated algebra in case 2(a) of the above theorem is not generated by two of the $1+n$ axes given.  In this case, $\sg_a \neq 1$ and $Z$ and $Z^{\sg_a}$ are disjoint.  The algebra is then generated by, for example, $b^{\sg_a}$ and $c$.
\end{remark}

\begin{remark}
The above theorem implies that the vast majority of shapes on the axet $X_1(1+n)$ collapse.  In particular, if $A_Y \in \cJ^2(\alpha) -\{ 2\B\}$ and $n \neq 3$, then all the shapes on $X_1(1+n)$ collapse unless the subalgebras for $A_Y$ and $A_Z$ are (quotients of) those found in the completions listed.  Explicitly, these shapes are:
\begin{enumerate}[leftmargin=1.5cm]
\item[(a)(i)] If $A \cong 6\A(\al, \tfrac{\al^2}{4(1-2\al)})$, then $A_Y \cong 3\C(\al)$ and $A_Z \cong 3\A(\alpha, \tfrac{\al^2}{4(1-2\al)})$.
\item[(a)(ii)] If $A \cong 6\J(2\bt,\bt)$, then $A_Y \cong 3\C(2\bt)$ and $A_Z \cong 3\C(\bt)$.
\item[(a)(iii)] If $A \cong \IY_3(\al, \tfrac{1}{2}, \mu)$, then $A_Y \cong 3\C(\al)$ and $A_Z \cong \IY_3(\al, \tfrac{1}{2}, -\mu)$.
\item[(a)(iv)] If $A \cong \cH_{2n}^J$, then $A_Y \cong 3\C(2)$ and $A_Z \cong \cH_n^J$ (and $(\al,\bt) = (2, \frac{1}{2})$).
\item[(a)(v)] If $A \cong \hatH_{2n}^J$, then $A_Y \cong 3\C(2)$ and $A_Z \cong \hatH_n^J$ (and $(\al,\bt) = (2, \frac{1}{2})$).
\item[(b)] If $A \cong \mathrm{Bar}_{0,1}(\frac{1}{2}, 2)$, then $A_Y \cong S(-2)$, $A_Z \cong 3\C(2)$.
\end{enumerate}
\end{remark}

Let us see what happens in the only case, $\cM(\frac{1}{4}, \frac{1}{32})$, where we do know all the $2$-generated algebras.  These are the Norton-Sakuma algebras \cite{IPSS}.  Since $n \geq 3$ is odd, it is either $3$, or $5$ and so in the shape $\Theta$, $\lla b,c \rra \cong 3\A, 3\C, 5\A$ and $\lla a, b \rra \cong 2\A, 2\B$.  Hence the only possible shapes are $3\A2\B$, $3\C2\B$, $5\A2\B$, $3\A2\A$, $3\C2\A$, or $5\A2\A$.  The completions of the first three are direct sum algebras in case (1) of Theorem \ref{1+nshape}.

\begin{corollary}\label{no3C2A}\textup{\cite[Proposition 5.2 and Corollary 5.3]{3gen4trans}}
\begin{enumerate}
\item The shape $3\A2\A$ has the unique completion $6\A$.
\item The shapes $3\C2\A$ and $5\A 2\A$ collapse.
\end{enumerate}
\end{corollary}
\begin{proof}
By Theorem \ref{1+nshape}, such a completion for any of these shapes would be $2$-generated with axet $X(2n)$, $n \geq 3$ odd.  By inspection, $6\A$ is the only such algebra and $\lla b, c \rra$ is isomorphic to $3\A$ and not $3\C$, or $5\A$.
\end{proof}

\end{document}